\documentclass
[
    a4paper,
    USenglish,
    cleveref,
    autoref,
    thm-restate,
]
{lipics-v2019}

\usepackage{csquotes} 

\usepackage{upgreek}
\usepackage{dsfont}        

\usepackage{booktabs}  

\usepackage
[
    ruled,         
    vlined,        
    linesnumbered, 
]
{algorithm2e} 

\usepackage{xspace}   

\usepackage{xparse}    
\usepackage
[
    textsize = scriptsize, 
]
{todonotes}            

\usepackage{algorithmicx}

\usepackage{nameref}

\usepackage{mathtools}

\makeatletter
    \def\IfEmptyTF#1%
    {%
        \if\relax\detokenize{#1}\relax%
            \expandafter\@firstoftwo%
        \else%
            \expandafter\@secondoftwo%
        \fi%
    }
\makeatother

\NewDocumentCommand{\mathOrText}{m}
{%
    \ensuremath{#1}\xspace%
}

\let\originalleft\left
\let\originalright\right
\renewcommand{\left}{\mathopen{}\mathclose\bgroup\originalleft}
\renewcommand{\right}{\aftergroup\egroup\originalright}

\makeatletter
    \DeclareRobustCommand{\bfseries}%
    {%
        \not@math@alphabet\bfseries\mathbf%
        \fontseries\bfdefault\selectfont%
        \boldmath%
    }

\xspaceaddexceptions{]\}}

\allowdisplaybreaks

\crefname{ineq}{inequality}{inequalities}
\creflabelformat{ineq}{#2{\upshape(#1)}#3}

\crefname{term}{term}{terms}
\creflabelformat{term}{#2{\upshape(#1)}#3}

\crefname{cond}{condition}{conditions}
\creflabelformat{term}{#2{\upshape(#1)}#3}

\crefname{assume}{assumption}{assumptions}
\creflabelformat{term}{#2{\upshape(#1)}#3}

\NewDocumentCommand{\functionTemplate}{m m m m o}%
{%
    \IfNoValueTF{#5}%
    {%
        \mathOrText{#1\left#2{#4}\right#3}%
    }%
    {%
        \mathOrText{#1#5#2{#4}#5#3}%
    }%
}

\newcommand*{\leftBracketType}{(}
\newcommand*{\rightBracketType}{)}

\NewDocumentCommand{\createFunction}{m m o o}%
{%
    \renewcommand*{\leftBracketType}{\IfNoValueTF{#3}{(}{#3}}%
    \renewcommand*{\rightBracketType}{\IfNoValueTF{#4}{)}{#4}}%
    \NewDocumentCommand{#1}{o o}%
    {%
        \IfNoValueTF{##1}%
        {%
            \mathOrText{#2}%
        }%
        {%
            \functionTemplate{#2}{\leftBracketType}{\rightBracketType}{##1}[##2]%
        }%
    }%
}

\DeclareDocumentCommand{\probabilisticFunctionTemplate}{m m O{} o}
{%
    \functionTemplate{#1}%
    {\lbrack}%
    {\rbrack}%
    {#2\IfEmptyTF{#3}{}{\ \IfNoValueTF{#4}{\left}{#4}\vert\ \vphantom{#2}#3\IfNoValueTF{#4}{\right.}{}}}%
    [#4]%
}

\newcommand*{\N}{\mathOrText{\mathds{N}}}

\newcommand*{\R}{\mathOrText{\mathds{R}}}

\newcommand*{\indicatorFunctionSymbol}{\mathds{1}}

\RenewDocumentCommand{\Pr}{m O{} o}%
{%
    \probabilisticFunctionTemplate{\mathrm{Pr}}{#1}[#2][#3]%
}

\NewDocumentCommand{\E}{m O{} o}%
{%
    \probabilisticFunctionTemplate{\mathrm{E}}{#1}[#2][#3]%
}

\NewDocumentCommand{\Var}{m O{} o}%
{%
    \probabilisticFunctionTemplate{\mathrm{Var}}{#1}[#2][#3]%
}

\DeclareDocumentCommand{\bigO}{m o}%
{%
    \functionTemplate{\mathrm{O}}{(}{)}{#1}[#2]%
}

\DeclareDocumentCommand{\smallO}{m o}%
{%
    \functionTemplate{\mathrm{o}}{(}{)}{#1}[#2]%
}

\DeclareDocumentCommand{\bigTheta}{m o}%
{%
    \functionTemplate{\Theta}{(}{)}{#1}[#2]%
}

\DeclareDocumentCommand{\bigOmega}{m o}%
{%
    \functionTemplate{\Omega}{(}{)}{#1}[#2]%
}

\DeclareDocumentCommand{\smallOmega}{m o}%
{%
    \functionTemplate{\upomega}{(}{)}{#1}[#2]%
}

\newcommand*{\circlePi}{\mathOrText{\uppi}}

\DeclareDocumentCommand{\eulerE}{o}%
{%
    \mathOrText{\mathrm{e}\IfNoValueTF{#1}{}{^{#1}}}%
}

\DeclareDocumentCommand{\poly}{m o}%
{%
    \functionTemplate{\mathrm{poly}}{(}{)}{#1}[#2]%
}

\createFunction{\id}{\mathrm{id}}

\NewDocumentCommand{\ind}{m o o}%
{%
    \IfNoValueTF{#2}%
    {%
        \mathOrText{\indicatorFunctionSymbol_{#1}}%
    }%
    {%
        \functionTemplate{\indicatorFunctionSymbol_{#1}}{(}{)}{#2}[#3]%
    }%
}

\DeclareDocumentCommand{\dom}{m o}%
{%
    \functionTemplate{\mathrm{dom}}{(}{)}{#1}[#2]%
}

\DeclareDocumentCommand{\rng}{m o}%
{%
    \functionTemplate{\mathrm{rng}}{(}{)}{#1}[#2]%
}

\DeclareDocumentCommand{\d}{o}%
{%
    \mathrm{d}\IfNoValueTF{#1}{}{^{#1}}%
}

\DeclareDocumentCommand{\set}{m m o}%
{%
    \mathOrText{\IfNoValueTF{#3}{\left}{#3}\{#1\ \IfNoValueTF{#3}{\left}{#3}\vert\ \vphantom{#1}#2\IfNoValueTF{#3}{\right.}{}\IfNoValueTF{#3}{\right}{#3}\}}%
}

\newtheorem{condition}[theorem]{Condition}
\newtheorem{observation}[theorem]{Observation}

\crefname{prop}{proposition}{propositions}
\creflabelformat{prop}{#2{\upshape(#1)}#3}

\newcommand*{\namedAppendix}[1]{Appendix: #1}
\newcommand*{\proofOf}[1]{Proof of #1}

\newcommand\numberthis{\addtocounter{equation}{1}\tag{\theequation}}

\newcommand*{\polymer}{\mathOrText{\gamma}}
\newcommand*{\polymerModel}{\mathOrText{\mathcal{P}}}
\newcommand*{\polys}{\mathOrText{\mathcal{C}}}
\newcommand*{\polyGraph}{\mathOrText{(\polys,\ncomp)}}
\newcommand*{\restrictedPolys}{\mathOrText{\mathcal{B}}}
\DeclareMathOperator*{\ncomp}{\not\sim}
\newcommand*{\gibbs}{\mathOrText{\mu}}
\newcommand*{\prt}{\mathOrText{Z}}
\newcommand*{\polyfam}{\mathOrText{\Gamma}}
\newcommand*{\clique}{\mathOrText{\Lambda}}
\newcommand*{\numberOfCliques}{\mathOrText{m}}
\newcommand*{\transitionProbabilitiesName}{\mathOrText{P}}

\newcommand*{\powerset}[1]{\mathOrText{2^{#1}}}
\newcommand*{\dtv}[2]{\mathOrText{d_{\text{TV}}\big(#1, #2\big)}}
\newcommand*{\size}[1]{\mathOrText{\left| #1 \right|}}
\newcommand*{\absolute}[1]{\mathOrText{\left| #1 \right|}}

\newcommand*{\err}{\mathOrText{\varepsilon}}

\newcommand*{\mix}[2]{\mathOrText{\tau_{#1} \left( #2 \right)}}
\newcommand*{\mixFunction}[1]{\mathOrText{\tau_{#1}}}

\newcommand*{\pmc}{clique dynamics condition\xspace}

\newcommand*{\ctc}{clique truncation condition\xspace}

\newcommand*{\fpc}{Fern\'{a}ndez--Procacci condition\xspace}

\DeclareDocumentCommand{\weight}{o}%
{%
    \mathOrText{w\IfNoValueTF{#1}{}{_{#1}}}%
}

\DeclareDocumentCommand{\polyfams}{o o}%
{%
    \mathOrText{\mathcal{F}\IfNoValueF{#1}%
        {%
            \IfNoValueF{#2}{_{\vert#2}}\IfEmptyTF{#1}{}{^{(#1)}}%
        }%
    }%
}

\DeclareDocumentCommand{\partitionFunction}{o o}%
{%
    \mathOrText{\prt\IfNoValueF{#1}%
        {%
            \IfNoValueF{#2}{_{\vert#2}}\IfEmptyTF{#1}{}{(#1)}%
        }%
    }%
}

\DeclareDocumentCommand{\gibbsDistribution}{o o}%
{%
\mathOrText{\gibbs\IfNoValueF{#1}%
    {%
        \IfNoValueF{#2}{_{\vert#2}}\IfEmptyTF{#1}{}{^{(#1)}}%
    }%
}%
}

\DeclareDocumentCommand{\gibbsDistributionFunction}{m o o}%
{%
    \IfNoValueTF{#2}%
    {%
        \gibbsDistribution%
    }%
    {%
        \IfNoValueTF{#3}{\gibbsDistribution[#2]}{\gibbsDistribution[#2][#3]}%
    }%
    \left(#1\right)%
}

\DeclareDocumentCommand{\polymerClique}{o}%
{%
    \mathOrText{\clique\IfNoValueTF{#1}{}{_{#1}}}%
}

\DeclareDocumentCommand{\markov}{o}%
{%
    \mathOrText{\mathcal{M}\IfNoValueTF{#1}{}{\left(#1\right)}}%
}

\DeclareDocumentCommand{\transitionProbabilities}{m m}%
{%
    \mathOrText{\transitionProbabilitiesName\left(#1, #2\right)}%
}

\newcommand*{\Uclique}[1]{\mathOrText{K_{#1}}}

\DeclareDocumentCommand{\prtFrac}{o}%
{%
    \mathOrText{\sigma\IfNoValueTF{#1}{}{_{#1}}}%
}

\DeclareDocumentCommand{\appxPrtFrac}{o}%
{%
    \mathOrText{\widehat{\sigma\IfNoValueTF{#1}{}{_{#1}}}}%
}

\newcommand*{\numSamples}{\mathOrText{s}}

\DeclareDocumentCommand{\particles}{o}%
{%
    \mathOrText{Q\IfNoValueTF{#1}{}{^{(#1)}}}%
}

\DeclareDocumentCommand{\potential}{o o}%
{%
\mathOrText{p\IfNoValueF{#1}%
    {%
        \IfNoValueF{#2}{^{(#2)}}\IfEmptyTF{#1}{}{_{#1}}%
    }%
}%
}

\DeclareDocumentCommand{\radius}{o o}%
{%
\mathOrText{r\IfNoValueF{#1}%
    {%
        \IfNoValueF{#2}{^{(#2)}}\IfEmptyTF{#1}{}{_{#1}}%
    }%
}%
}

\DeclareDocumentCommand{\grid}{o}%
{%
    \mathOrText{G\IfNoValueTF{#1}{}{^{(#1)}}}%
}

\DeclareDocumentCommand{\subgrid}{o}%
{%
    \mathOrText{H\IfNoValueTF{#1}{}{_{#1}}}%
}

\DeclareDocumentCommand{\hardSpherePrt}{o o}%
{%
    \mathOrText{\prt\IfNoValueF{#1}%
        {%
            \IfNoValueF{#2}{_{#2}}\IfEmptyTF{#1}{}{\left(#1\right)}%
        }%
    }%
} 

\DeclareDocumentCommand{\valid}{o o}%
{%
    \mathOrText{D\IfNoValueF{#1}%
        {%
            \IfEmptyTF{#1}{}{_{#1}}\IfNoValueF{#2}{\left(#2\right)}%
        }%
    }%
}

\DeclareDocumentCommand{\lebesgue}{m o}%
{%
    \mathOrText{\nu^{#1} \IfNoValueTF{#2}{}{\left(#2\right)}}%
}

\DeclareDocumentCommand{\rescale}{o o}%
{%
    \mathOrText{\varphi\IfNoValueF{#1}%
        {%
            \IfEmptyTF{#1}{}{^{(#1)}}\IfNoValueF{#2}{\left(#2\right)}%
        }%
    }%
} 

\DeclareDocumentCommand{\remap}{o o}%
{%
    \mathOrText{\Phi\IfNoValueF{#1}%
        {%
            \IfEmptyTF{#1}{}{^{(#1)}}\IfNoValueF{#2}{\left(#2\right)}%
        }%
    }%
}

\DeclareDocumentCommand{\integerSphere}{m o}%
{%
    \mathOrText{b_{#1} \IfNoValueTF{#2}{}{\left(#2\right)}}%
}

\newcommand*{\indicator}[1]{\mathOrText{\indicatorFunctionSymbol{\left\{#1\right\}}}}

\newcommand*{\fugacity}{\mathOrText{\lambda}}

\newcommand*{\leftPart}{\mathOrText{\text{L}}}
\newcommand*{\rightPart}{\mathOrText{\text{R}}}

\newcommand*{\partition}[1]{\mathOrText{V_{#1}}}

\DeclareDocumentCommand{\neighbors}{m o}%
{%
    \mathOrText{N\IfNoValueTF{#2}{}{_{#2}}}\left( #1 \right)%
}

\newcommand*{\polymerModelPart}[1]{\mathOrText{\polymerModel^{\left(#1\right)}}} 

\newcommand*{\polysPart}[1]{\mathOrText{\polys^{\left(#1\right)}}}

\DeclareDocumentCommand{\weightPart}{m o}%
{%
    \mathOrText{\IfNoValueTF{#2}{\weight^{\left(#1\right)}}{\weight[#2]^{\left(#1\right)}}}%
}

\newcommand*{\polyVertex}[1]{\mathOrText{\overline{#1}}}

\DeclareDocumentCommand{\partGibbs}{m o}%
{%
    \mathOrText{\IfNoValueTF{#2}{\gibbsDistribution[#1]}{\gibbsDistribution[#1][#2]}}%
}

\DeclareDocumentCommand{\partPrt}{m o}%
{%
    \mathOrText{\IfNoValueTF{#2}{\partitionFunction^{(#1)}}{\partitionFunction^{(#1)}[][#2]}}%
}

\DeclareDocumentCommand{\hcPartitionFunction}{m o}%
{%
    \mathOrText{\IfNoValueTF{#2}{\partitionFunction \left( #1 \right)}{\partitionFunction \left( #2, #1 \right)}}%
}

\newcommand*{\Degree}{\mathOrText{\Delta}}

\DeclareDocumentCommand{\grid}{o}%
{%
    \mathOrText{G\IfNoValueTF{#1}{}{^{(#1)}}}%
}

\DeclareDocumentCommand{\transform}{o}%
{%
    \mathOrText{\overline{\ln}\IfNoValueTF{#1}{}{\left( #1 \right)}}%
}

\DeclareDocumentCommand{\filtration}{o}%
{%
    \mathOrText{\mathcal{F}\IfNoValueTF{#1}{}{_{#1}}}%
}

\title{Polymer Dynamics via Cliques: New Conditions for Approximations} 

\titlerunning{Polymer Dynamics via Cliques} 

\author{Tobias Friedrich}{Hasso Plattner Institute, University of Potsdam, Potsdam, Germany}{tobias.friedrich@hpi.de}{https://orcid.org/0000-0003-0076-6308}{} 

\author{Andreas Göbel}{Hasso Plattner Institute, University of Potsdam, Potsdam, Germany}{andreas.goebel@hpi.de}{}{}

\author{Martin~S. Krejca}{Hasso Plattner Institute, University of Potsdam, Potsdam, Germany}{martin.krejca@hpi.de}{https://orcid.org/0000-0002-1765-1219}{}

\author{Marcus Pappik}{Hasso Plattner Institute, University of Potsdam, Potsdam, Germany}{marcus.pappik@hpi.de}{}{}

\authorrunning{T. Friedrich and A. Göbel and M.~S. Krejca and M. Pappik} 

\Copyright{Tobias Friedrich and Andreas Göbel and Martin~S. Krejca and Marcus Pappik} 

\ccsdesc[500]{Theory of computation~Random walks and Markov chains} 

\keywords{Markov chain, partition function, Gibbs distribution, approximate counting, abstract polymer model} 

\category{} 

\relatedversion{} 

\supplement{} 

\acknowledgements{}

\nolinenumbers 

\hideLIPIcs  

\begin{document}

\maketitle
\begin{abstract}
Abstract polymer models are systems of weighted objects, called polymers, equipped with an incompatibility relation. An important quantity associated with such models is the partition function, which is the weighted sum over all sets of compatible polymers. Various approximation problems reduce to approximating the partition function of a polymer model. Central to the existence of such approximation algorithms are weight conditions of the respective polymer model. Such conditions are derived either via complex analysis or via probabilistic arguments. We follow the latter path and establish a new condition---the \pmc---, which is less restrictive than the ones in the literature. We introduce a new Markov chain where the \pmc implies rapid mixing by utilizing cliques of incompatible polymers that naturally arise from the translation of algorithmic problems into polymer models. This leads to improved parameter ranges for several approximation algorithms, such as a factor of at least $2^{1/\alpha}$ for the hard-core model on bipartite $\alpha$-expanders.
\end{abstract}
\newpage

\section{Introduction} \label{sec:intro}

Statistical physics models systems of interacting particles as probability distributions. This approach explains a variety of real-world phenomena, including ferromagnetism~\cite{Ising}, segregation~\cite{SchellingBook}, and real-world network generation~\cite{StatPhysNetworks}. A characteristic of such systems is that they undergo phase transitions depending on some external parameter. Such phase transitions have been recently linked with 
the tractability of computational tasks and have lead to a two-way exchange: tools from statistical physics are used to explain computational phenomena, and tools from computer science are used to explain physical phenomena. An established technique  for investigating phase transitions in statistical physics that involves translating the states of a spin system as perturbations from a ground state~\cite[Chapter 7]{friedli_velenik_2017} has been recently introduced to computer science as an algorithmic tool for computational tasks of spin systems~\cite{HPR19}.

To motivate the definition of the central mathematical object of this article, we give a high-level description of how to model a spin system in terms of perturbations from a ground state.
Assume we study a $q$-state spin system on a graph $G$. The states of the spin system are usually mappings $\sigma\colon V(G)\rightarrow Q$ from the vertices of~$G$ to some finite set~$Q$. Each such configuration~$\sigma$ has a weight $w(\sigma)\in \R_{\geq 0}$ and the sum of the weights of all the configurations $Z=\sum_\sigma(w(\sigma))$ is called the \emph{partition function.} For each configuration~$\sigma$, the probability distribution that characterizes our system yields $\mu(\sigma)=w(\sigma)/Z$. Let $\sigma_0$ be the ground state we use in this translation. Given a configuration $\sigma$, we identify the set of vertices~$D\subseteq V(G)$ where, for each $v\in D$, we have $\sigma_0(v)\neq\sigma(v)$. Observe that we can uniquely identify this configuration by a set $\Gamma$ whose elements $\gamma$ consist of a connected component of $G[D]$ together with the restriction of $\sigma$ on this component.
Furthermore, we assign a weight $w_\gamma$ to each $\gamma\in\Gamma$, such that $\prod_{\gamma\in\Gamma}w_\gamma=w(\sigma)/w(\sigma_0)$. Thus, provided that all such sets of pairs $\Gamma$ contain no two pairs~$\gamma$, $\gamma'$ that are incompatible, i.e. $\Gamma$ cannot be uniquely decoded to an assignment because for example~$\gamma$ and~$\gamma'$ map the same vertex to a different element in~$Q$, there is a bijection between the configurations~$\sigma$ and the sets $\Gamma$. Furthermore, the distribution $\mu$ is expressed as a distribution over the sets~$\Gamma$, since it retains the property that the probability of $\Gamma$ is proportional to its weight. Such a construction suggests the following definition.

A \emph{polymer model} $\polymerModel=(\polys, \weight, \ncomp)$ is a tuple consisting of a non-empty, countable set~\polys, a set $\weight = \{\weight[\polymer]\}_{\polymer \in \polys}$ of positive real weights and a reflexive and symmetric relation $\ncomp \subseteq \polys^2$. The elements $\polymer \in \polys$ are called \emph{polymers}. The relation $\ncomp$ is called the \emph{incompatibility} relation and, for $\polymer, \polymer' \in \polys$, we say that \polymer and~$\polymer'$ are \emph{incompatible} if $\polymer \ncomp \polymer'$, and that they are \emph{compatible} otherwise.
We call a \emph{finite} subset $\polyfam \subseteq \polys$ a \emph{polymer family} if and only if all polymers of~\polyfam are pairwise compatible. Given a polymer model~\polymerModel, let~\polyfams[\polymerModel] denote the set of \emph{all} polymer families of~\polymerModel. Note that~\polyfams[\polymerModel] is countable. 
The \emph{partition function} of~\polymerModel is defined as
\begin{equation}
    \label{eq:partitionFunction}
    \partitionFunction[\polymerModel] = \sum\nolimits_{\polyfam \in \polyfams[\polymerModel]} \prod\nolimits_{\polymer \in \polyfam} \weight[\gamma] ,
\end{equation}
which we require to be finite.
Further, the \emph{Gibbs distribution} of~\polymerModel is the probability distribution~$\gibbsDistribution[\polymerModel]$ over~\polyfams[\polymerModel] such that, for all $\polyfam \in \polyfams[\polymerModel]$,
\begin{equation}
    \label{eq:gibbsDistribution}
    \gibbsDistributionFunction{\polyfam}[\polymerModel] = \frac{\prod_{\polymer \in \polyfam} \weight[\polymer]}{\partitionFunction[\polymerModel]} .
\end{equation}

A helpful interpretation for understanding the definition of a polymer model is the following.
Ignoring the reflexivity of $\ncomp$, we view the pair~$\polyGraph$ as a graph, which we call the \emph{polymer graph}. We observe that the families of $\polyfams[\polymerModel]$ correspond to the independent sets of $\polyGraph$. Thus, for the special case where $\weight[\polymer]=\lambda\in\R$, for each $\gamma\in\polys$, the distribution $\gibbsDistribution$ is the hard-core model~\cite{2006:Weitz:counting_independent_sets} on the polymer graph and $\partitionFunction[\polymerModel]$ is the independence polynomial~\cite{2017:Peters:conjecture_of_Sokal}. 

We aim to maximize the parameter range for which the following two computational tasks can be done in polynomial time.
\begin{enumerate}[(1)]
    \item Approximately sampling from the Gibbs distribution of a polymer model, i.e., return a random family~$\Gamma$ from a distribution with total-variation distance of at most~$\varepsilon$ from $\gibbsDistribution[\polymerModel]$.
    \item Returning an estimate $\widetilde Z$, such that $(1-\varepsilon)\partitionFunction[\polymerModel]\leq\widetilde Z\leq(1+\varepsilon)\partitionFunction[\polymerModel]$.
\end{enumerate}

\subsection{Known Algorithmic Results}\label{sec:related_work}

There is an expanding list of results that utilize abstract polymer models to obtain efficient approximation and sampling algorithms for new parameter regimes for various spin systems on graphs. This line of research was initiated by Helmuth et~al.~\cite{HPR19}, who used polymers to obtain polynomial-time approximation and sampling algorithms at a regime where the weight of the interactions of particles with an external field is low. For problems that are hard to approximate on general inputs, polynomial-time approximation algorithms are derived by utilizing restrictions upon the input graph of a spin system. Examples include spin systems on expander graphs~\cite{GGS20,JKP19,LLLM19}, the hard-core model on unbalanced bipartite graphs~\cite{CP20}, and the ferromagnetic Potts model on $d$-dimensional lattices~\cite{BCHPT20}. Polymer models have also been used to approximate and sample edge spin systems (holant problems) at low temperatures~\cite{CFFGL19}.

Translating a spin system on a graph $G$ with $n$ vertices into an abstract polymer model commonly results in an exponential number of polymers in terms of~$n$---as can be observed in our initial example.
However, we are interested in approximation and sampling algorithms with a runtime polynomial in $n$.
There are two main approaches for such algorithms.

\paragraph*{(i) Cluster Expansion} This approach considers complex weights for the polymers and is based on an infinite series expansion of $\ln \partitionFunction$ (the \emph{cluster expansion}).
The essential element for polynomial-time computation is a theorem of Koteck{\'y} and Preiss~\cite[Theorem~1]{1986:Kotecky:cluster_expansion_polymer_models}, a condition for establishing absolute convergence of the cluster expansion. By satisfying this condition, the cluster expansion is truncated to its most significant terms, obtaining an $\varepsilon$-additive approximation for $\ln \partitionFunction$. Computing the significant terms is achieved by enumerating connected induced subgraphs of the polymer graph of size up to $\log|\polys|$. By an algorithm of Patel and Regts, the enumeration takes polynomial time in terms of the input graph of the spin system~\cite{PR17}. The $\varepsilon$-additive approximation of $\ln \partitionFunction$ immediately gives a multiplicative $\varepsilon$-approximation for~$\partitionFunction$. The runtime of this approach is commonly $n^{\bigO{\log\Delta}}$, where $\Delta$ is the maximum degree of the input graph $G$ for the spin system and $n=|V(G)|$.
Approximating~$\partitionFunction$ together with the self-reducibility of the polymer model yields a sampling algorithm for $\gibbsDistribution[\polymerModel]$~\cite{HPR19}. 

\paragraph*{(ii) Markov Chain Monte Carlo}
Initiated by Chen et~al.~\cite{CGGPSV19}, this approach defines a Markov chain with state space~$\polyfams[\polymerModel]$ and stationary distribution~$\gibbsDistribution[\polymerModel]$. The Markov chain requires the polymer model to have originated from a spin system on a graph $G$ with $n$ vertices. Iteratively, the chain samples a polymer~$\gamma$ with probability proportional to its weight~$\weight[\gamma]$ and then adds or removes $\gamma$ from its state if possible. When the \emph{mixing condition}~\cite[Definition~1]{CGGPSV19} is satisfied, the Markov chain converges to $\gibbsDistribution[\polymerModel]$ after $\bigO{n\log n}$ iterations. The mixing condition matches a convergence condition arising from an analysis by Fernández et~al.~\cite{FFG01} of another stochastic process of polymers on lattices. An $\varepsilon$-approximate sampler for~$\gibbsDistribution[\polymerModel]$ is obtained by simulating the Markov chain. The computational challenge for this approach is to sample the polymer~$\gamma$ in order to perform a transition of the Markov chain.  As Chen et~al.~\cite{CGGPSV19} show, this can be done in expected constant time provided the \emph{sampling condition}~\cite[Definition~4]{CGGPSV19} is satisfied. This results in an $\bigO{n\log n}$ algorithm for sampling from the Gibbs distribution of a spin system. Using Simulated Annealing, Chen et al. convert this sampler into a randomized approximation scheme (FPRAS) for $\partitionFunction$ that runs in expected $\bigO{n^2\log n}$ time.

\paragraph*{Comparison of Established Conditions for the Methods Above}
A number of conditions for the convergence of the cluster expansion have appeared in the literature~\cite{Dob96,FP07,1986:Kotecky:cluster_expansion_polymer_models}. The condition of Fernández and Procacci~\cite{FP07} is the least restrictive among them, i.e., the other conditions imply it. Thus, using the Fern\'andez--Procacci condition, one could potentially obtain approximation algorithms for broader parameter ranges than the ones obtained by using, e.g., the Koteck{\'y}--Preiss condition~\cite{1986:Kotecky:cluster_expansion_polymer_models}.
However, the condition by Kotecký and Preiss is convenient to apply in polymer models of vertex spin systems and comes with implications on the rate of convergence of the cluster expansion used in algorithmic settings. When compared to cluster expansion conditions (restricted to non-negative real weights), the mixing condition of Chen et~al.~\cite{CGGPSV19} is less restrictive than the Koteck{\'y}--Preiss condition, however, it is incomparable with the Fernández--Procacci condition. Note that the fast run-times of Chen et al.~\cite{CGGPSV19} are dependent on the sampling condition, which imposes the largest restriction on the parameter range of the applications it is used for.


\subsection{Our Results}

We study a new Markov chain~$(X_t)_{t\in\N}$ for abstract polymer models with stationary distribution~$\gibbsDistribution[\polymerModel]$.
The dynamics of our Markov chain are based on a \emph{clique cover,} i.e., a set $\polymerClique = \{\polymerClique[i]\}_{i \in [\numberOfCliques]}$ with $\bigcup \polymerClique = \polys$ such that the polymers in each clique~$\polymerClique[i]$ are pairwise incompatible. Note that families of compatible polymers contain at most one polymer per clique. At each step, our Markov chain chooses $i \in [\numberOfCliques]$  uniformly at random and samples a family in $\polymerClique[i]$ according to the distribution $\gibbsDistribution[][\clique_i]$ defined as follows. For~$\gamma\in\polymerClique[i]$, we have $\gibbsDistributionFunction{\{\gamma\}}[][\clique_i] = \weight[\gamma]/\partitionFunction[][\polymerClique[i]]$ and, for the empty set, $\gibbsDistributionFunction{\emptyset}[][\clique_i] = 1/\partitionFunction[][\polymerClique[i]]$, where $\partitionFunction[][\polymerClique[i]]=1 + \sum_{\polymer \in \clique_i} \weight[\polymer]$. If the empty family is chosen and $X_t$ contains a polymer from $\polymerClique[i]$, then the chain removes this polymer. If the family chosen contains a polymer, then, if possible, the chain adds this polymer to its state. For a detailed description of our chain, please refer to \Cref{def:markov_chain}.

Our chain can be viewed as a natural generalization of the (spin) Glauber dynamics (cf.~\cite[insert/delete chain]{DyerG00}). Any abstract polymer model has a trivial clique cover, where each clique contains exactly one polymer. From the choice of sampling from $\gibbsDistribution[][\clique_i]$ for each clique chosen at each iteration, the clique dynamics with the trivial clique cover coincides with the Glauber dynamics.
Clique covers with a much smaller number of cliques than the number of vertices arise naturally from the translation of spin systems into polymer models. For example, the translation we discussed earlier in the introduction yields a clique cover with~$n$ cliques, one for each vertex in the original graph $G$. Since the number of polymers is commonly exponential in the size of $G$, our chain utilizes that a family of compatible polymers may contain at most one polymer from each polymer clique. The chain of Chen et al.~\cite{CGGPSV19} also utilizes this fact, however, in a more restricted setting and with a different sampling distribution for each vertex clique. 

Central to our mixing time analysis for this chain is the following condition.
\begin{condition}[Clique Dynamics]
    \label{def:pmc}
    Let $\polymerModel = (\polys, \weight, \ncomp)$ be a polymer model, and let $f\colon \polys \to \R_{>0}$. We say that~\polymerModel \emph{satisfies the \pmc{} with~$f$} if and only if, for all $\polymer \in \polys$,
    \[
    \sum\nolimits_{\substack{\polymer' \in \polys\colon \polymer' \ncomp \polymer, \\ \polymer' \neq \polymer}} \Big(f(\polymer') \frac{\weight[\polymer']}{1 + \weight[\polymer']}\Big) \le f(\gamma)\ .
    \]
\end{condition}

We show that when the \pmc is satisfied, the mixing time of our Markov chain is polynomial in the number of cliques in the clique cover and logarithmic in~$f$ (\Cref{thm:markov_chain}).
When restricted to the setting of Chen et~al.~\cite{CGGPSV19}, the \pmc is implied by the mixing condition and thus less restrictive. The function~$f$ in our condition makes it easily comparable with the conditions for cluster expansion. Unlike the Koteck{\'y}--Preiss condition~\cite{1986:Kotecky:cluster_expansion_polymer_models}, the \pmc allows for a more flexible choice of $f$ in order to obtain a greater parameter range for the algorithmic applications (cf. \cite[Remark 2.4]{CFFGL19}).

We further show that the \pmc{} is more general than the Fern\'andez--Procacci~\cite{FP07} condition for the cluster expansion---and consequently more general than the Koteck{\'y}--Preiss condition~\cite{1986:Kotecky:cluster_expansion_polymer_models} (see \Cref{sec:comparison}). An implication of our analysis is that cluster expansion conditions imply our condition for the mixing time of clique dynamics. To the best of our knowledge, this is the first connection between cluster expansion and mixing times of Markov chains for polymer models and might be of independent interest to the statistical-physics community. This is in line with the special case of the hard-core model, where the parameter range on the reals has a bigger radius than on the complex plane~\cite{2017:Peters:conjecture_of_Sokal}.

The \pmc allows us to prove the following theorem for approximately sampling from $\gibbsDistribution$, which is our main algorithmic result.

\begin{restatable}{theorem}{samplingTheorem}
    \label{thm:sampling}
    Let $\polymerModel = (\polys, w, \ncomp)$ be a computationally feasible polymer model, let~\polymerClique be a polymer clique cover of~\polymerModel with size~\numberOfCliques, and let $\prt_{\max} = \max_{i \in [\numberOfCliques]} \{\partitionFunction[][\clique_i]\}$.
    Further, assume that
    \begin{enumerate}[(a)]
        \item \label[assume]{thm:sampling:pmc}
        $\polymerModel$ satisfies the \pmc for a function $f$ such that, for all $\gamma \in \polys$, it holds that $\eulerE^{-\poly{\numberOfCliques}} \le f(\gamma) \le \eulerE^{\poly{\numberOfCliques}}$,
        \item \label[assume]{thm:sampling:prt}
        $\prt_{\max} \in \poly{\numberOfCliques}$, and that,
        \item \label[assume]{thm:sampling:inner}
        for all $i \in [\numberOfCliques]$, we can sample from $\gibbsDistribution[][\clique_i]$ in time $\poly{\numberOfCliques}$.
    \end{enumerate}
    
    Then, for all $\err \in (0, 1]$, we can $\err$-approximately sample from~$\gibbs$ in time $\poly{\numberOfCliques/\err}$.
\end{restatable}

Additionally, as we discuss in \Cref{sec:approx_prt}, we use self-reducibility on the clique cover as well as \Cref{thm:sampling} to obtain an $\varepsilon$-approximation algorithm for the partition function $\partitionFunction$.

\begin{restatable}{theorem}{approximationPartitionFunction}
    \label{thm:appx_partition_function}
    Let $\polymerModel = (\polys, w, \ncomp)$ be a computationally feasible polymer model, let~\polymerClique be a polymer clique cover of~\polymerModel with size~\numberOfCliques.
    Assume that~$\polymerModel$ satisfies the conditions of \Cref{thm:sampling}.
    For all $\err \in (0, 1]$, there is a randomized $\err$-approximation of~\partitionFunction computable in time $\poly{\numberOfCliques/\err}$.
\end{restatable}

Since it is common for spin systems on graphs with $n$ vertices to translate into polymer models with a clique cover of $n$ cliques, \Cref{thm:sampling,thm:appx_partition_function} imply polynomial-time algorithms for their respective problems. Assumption~(\ref{thm:sampling:pmc}) allows for a broad range in the choice of~$f$ from the \pmc and assumption~(\ref{thm:sampling:prt}) is commonly satisfied when we chose the parameters in order to satisfy assumption~(\ref{thm:sampling:pmc}).
However, applying \Cref{thm:sampling,thm:appx_partition_function} to spin systems previously studied in the literature, assumption~(\ref{thm:sampling:inner}) is not straightforward to satisfy, as the size of the cliques are commonly exponential in $n=|V(G)|$. 
As we are interested in extending the parameter range while remaining in the realm of polynomial-time computations, we do not need to use such a restrictive condition. For this purpose, we introduce the \ctc.

\begin{condition}[Clique Truncation]
    \label{def:ctc}
    Let $\polymerModel = (\polys, w, \ncomp)$ be a polymer model, let~\polymerClique be a polymer clique cover of~\polymerModel with size~\numberOfCliques, and let $\size{\cdot}$ be a size function for~\polymerModel as in \Cref{def:size}.
    For all $i \in [\numberOfCliques]$, we say that~$\clique_i$ \emph{satisfies the \ctc{}} for a monotonically increasing, invertible function $g\colon \R \to \R_{>0}$ and a bound $B \in \R_{> 0}$ if and only if
    \[
    \sum\nolimits_{\gamma \in \clique_i} g(\size{\gamma}) w_{\gamma} \le B.
    \]
\end{condition}

We show that the \ctc{} allows to reduce the size of each clique to a polynomial in~$n$ by removing low weight polymers from the polymer model.\footnote{A similar idea was used for the hard-core model on bipartite expanders in the first arXiv version of~\cite{CGGPSV19}.} More precisely, \Cref{lemma:model_truncation} states that, for an $\varepsilon$-approximation, it is sufficient to consider only polymers~\polymer with $\size{\polymer}\leq g^{-1}(B\numberOfCliques/\varepsilon)$. This allows us to use the algorithm of Patel and Regts~\cite{PR17} to sample from the Gibbs distribution of each clique by enumerating all its polymers. In all our calculations, the parameter range restrictions imposed by the \ctc are weaker than the ones imposed by the \pmc. As illustrated in~\Cref{table:bounds}, this leads to improved parameter ranges for spin systems previously studied in literature.

\begin{table}
    \centering
    \caption{\label{table:bounds} Improvement on the parameter ranges of our technique for problems with known approximation algorithms. Note that for a fair comparison we refined the calculations of the bounds in~\cite{JKP19} in a similar fashion as in~\Cref{appendix:expanders}.}
     \begin{tabular}{p{4.2 cm}p{4.2 cm}l}
     \toprule
       \textbf{Problem} & \textbf{Previous range} & \textbf{New range}  \\
     \midrule
      {\small Hard-core model on\newline \hspace*{0.5 em}bipartite $\alpha$-expanders}  &  
      	$\lambda>(\eulerE^2 \Delta^2)^\frac{1}{\alpha}$~\cite{JKP19} & 
      	$\lambda\geq \left(\frac{\eulerE}{0.8} \Degree^2 \right)^{\frac{1}{\alpha}}$  \\
     {\small $q$-state Potts model on\newline \hspace*{0.5 em} $\alpha$-expanders}  &  
     	$\beta> \frac{9/4 + \ln(\Delta q)}{\alpha}$~\cite{JKP19} & 
     	$\beta \geq \frac{3/2 + \ln(\Delta q)}{\alpha}$  \\
      {\small Hard-core model on \newline \hspace*{0.5 em}unbalanced bipartite graphs}  &  
      	$6\Delta_{\leftPart}\Delta_{\rightPart}\lambda_{\rightPart}\leq(1+\lambda_{\leftPart})^\frac{\delta_{\rightPart}}{\Delta_{\leftPart}}$~\cite{CP20} & 
       	$3.3353 \Delta_{\leftPart}\Delta_{\rightPart}\lambda_{\rightPart}\leq(1+\lambda_{\leftPart})^\frac{\delta_{\rightPart}}{\Delta_{\leftPart}}$ \\
      {\small Perfect matching\newline \hspace*{0.5 em}polynomial}  &  
      	$z\leq \left(\sqrt{4.8572 \, (\Delta -1)}\right)^{-1}$~\cite{CFFGL19} &
      	$z\leq \left(\sqrt{2.8399 (\Delta -1)}\right)^{-1}$ \\
     \bottomrule
    \end{tabular}
\end{table} 

\subsection{Outline}

In \Cref{sec:prelim}, we establish notation and introduce the tool for bounding the mixing time of our chain. We define and analyze our Markov chain in \Cref{sec:polymerDynamics}. The algorithmic results are stated in \Cref{sec:algo}. Last, in \Cref{sec:truncation}, we show how to efficiently sample polymers from their respective cliques, which we use to improve the parameter ranges of known algorithmic bounds on spin systems.
Due to space limitations, all of our proofs are in the appendix.

\section{Preliminaries}\label{sec:prelim}
We denote the set of all natural numbers, including~$0$, by~\N and the set of all real numbers by~\R. For an $n \in \N$, let $[n] = [1, n] \cap \N$. If the polymer model~\polymerModel is clear from context, we may drop the index and write~\polyfams, \partitionFunction, and \gibbsDistribution instead of~\polyfams[\polymerModel], \partitionFunction[\polymerModel], and \gibbsDistribution[\polymerModel], respectively.

We use the following formal notion of approximate sampling.
Let $\nu$  be a probability distribution on a countable state space $\Omega$.
For $\err \in (0, 1]$, we say that a distribution $\xi$ on $\Omega$ is an \emph{$\err$-approximation of $\nu$} if and only if $\dtv{\nu}{\xi} \le \err$, where $\dtv{\cdot}{\cdot}$ denotes the total-variation distance.
Further, we say that we can \emph{$\err$-approximately sample from $\nu$} if and only if we can sample from any distribution $\xi$ such that $\xi$ is an $\err$-approximation of $\nu$. 

We are also interested in approximating the partition function of polymer models, which we define as follows.
For $x \in \R_{>0}$ and $\err \in (0, 1]$, we call a random variable $X$ a \emph{randomized $\err$-approximation for~$x$} if and only if
\[
	\Pr{(1 - \err) x \le X \le (1 + \err) x} \ge \frac{3}{4} .
\]
Note that if~$x$ is the output to an algorithmic problem on some instance and independent samples of $X$ can be obtained in polynomial time in the instance size and $1/\err$, then this translates to the definition of an FPRAS.

\subsection{Restricted Polymer Models}
\label{sec:restrictedPolymerModels}

We base the transitions of our Markov chain for a polymer model $(\polys, \weight, \ncomp)$ on restricted sets $\restrictedPolys \subseteq \polys$. We define the set of all polymer families \emph{restricted to~\restrictedPolys} to be $\polyfams[][\restrictedPolys] = \polyfams \cap \powerset{\restrictedPolys}$. Further, we define the restricted partition function~\partitionFunction[][\restrictedPolys] to be \cref{eq:partitionFunction} but with~\polyfams[\polymerModel] replaced by~\polyfams[][\restrictedPolys]. Similarly, we define the restricted Gibbs distribution~\gibbsDistribution[][\restrictedPolys] to be a probability distribution over~\polyfams[][\restrictedPolys], i.e., \cref{eq:gibbsDistribution} but with~\partitionFunction[\polymerModel] replaced by~\partitionFunction[][\restrictedPolys]. Our restrictions are special sets of polymers, which we define next.

By definition, for a polymer model, a polymer family~\polyfam cannot contain incompatible polymers. Thus, when considering a subset $\restrictedPolys \subseteq \polys$ where all polymers are pairwise incompatible, at most one polymer of~\restrictedPolys is in~\polyfam. We call such a subset~\restrictedPolys a \emph{polymer clique.}

Last, for an $\numberOfCliques \in \N_{> 0}$, we call a set $\polymerClique = \{\polymerClique[i]\}_{i \in [\numberOfCliques]}$ of polymer cliques a \emph{polymer clique cover} if and only if $\bigcup \polymerClique = \polys$, and we call~\numberOfCliques the \emph{size} of~\polymerClique. Note that the elements of~\polymerClique need not be pairwise disjoint. Further note that, for each $i \in [\numberOfCliques]$, the partition function restricted to~\polymerClique[i] boils down to
\begin{align*}
    \partitionFunction[][\polymerClique[i]] = \sum_{\polyfam \in \polyfams[][\clique_i]} \prod_{\polymer \in \polyfam} \weight[\polymer] = 1 + \sum_{\polymer \in \clique_i} \weight[\polymer] ,
\end{align*}
as the polymers of~\polymerClique[i] are pairwise incompatible and thus each family of~\polymerClique[i] (except~$\emptyset$) contains a single polymer. Similarly, the Gibbs distribution restricted to~\polymerClique[i] simplifies to $\gibbsDistributionFunction{\emptyset}[][\clique_i] = 1/\partitionFunction[][\polymerClique[i]] = 1/(1 + \sum_{\polymer \in \clique_i} \weight[\polymer])$ and, for each $\polymer' \in \clique_i$, to $\gibbsDistributionFunction{\{\polymer'\}}[][\clique_i] = \weight[\polymer']/\partitionFunction[][\polymerClique[i]] = \weight[\polymer']/(1 + \sum_{\polymer \in \clique_i} \weight[\polymer])$.

\subsection{Markov Chains}


For a Markov chain~\markov with a unique stationary distribution~$D$ and an $\varepsilon \in (0, 1]$, let~\mix{\markov}{\varepsilon} denote the mixing time of~\markov (with error~$\varepsilon$). That is, \mix{\markov}{\varepsilon} denotes the first point in time $t \in \N$ such that, for every initial state, the total-variation distance between~$D$ and the distribution of~\markov at time~$t$ is at most~$\varepsilon$.

In order to bound the mixing time of our Markov chains, we use a theorem by \cite[Theorem~$3.3$]{GreenbergPR09IncorrectSupermartingaleDrift}. Unfortunately, the theorem is not correct in its original formulation. Therefore, we provide an alternative formulation, which we use. We give a full proof of this theorem in \Cref{appendix:coupling_lemma}, where we also discuss why the original assumptions are insufficient.

\begin{restatable}{theorem}{expPotentialTheorem}
    \label{lemma:exp_potential}
    Let~$\markov$ be an ergodic Markov chain with state space $\Omega$ and with transition matrix $P$ such that, for all $x \in \Omega$, it holds that $P(x, x) > 0$.
    For $d, D \in \R_{> 0}$, $d \leq D$, let $\delta\colon \Omega^2 \to \{0\} \cup [d, D]$ be such that $\delta(x, y) = 0$ if and only if $x=y$.
    Assume that there is a coupling between the transitions of two copies $(X_t)_{t \in \N}$ and $(Y_t)_{t \in \N}$ of~$\markov$ such that, for all $t \in \N$ and all $x, y \in \Omega$, it holds that
    \begin{align}
        \label{eq:driftCondition}
        \E{\delta(X_{t+1}, Y_{t+1})}[X_t = x, Y_t = y] \le \delta(x, y).
    \end{align}
    Furthermore, assume that there are $\kappa, \eta \in (0, 1)$ such that, for the same coupling and all $t \in \N$ and all $x, y \in \Omega$ with $x \neq y$, it holds that
    \begin{align}
        \label{eq:largeJumps}
        \Pr{|\delta(X_{t+1}, Y_{t+1}) - \delta(x, y)| \ge \eta \delta(x, y)}[X_t = x, Y_t = y] \ge \kappa. 
    \end{align} 
    
    Then, for all $\varepsilon \in (0, 1]$, it holds that
    \begin{align*}
        \mix{\markov}{\err} &\leq \frac{\big(\ln(D/d) + 2 \ln(2)\big)^2}{\ln(1+\eta)^2 \kappa}  \ln \left (\frac{1}{\err} \right ) .
    \end{align*}
    If $\ln(D/d) \in \bigOmega{1}$, then this bound simplifies to
    \begin{align*}
        \mix{\markov}{\err} &\in \bigO{\frac{\ln(D/d)^2}{\ln(1+\eta)^2 \kappa}  \ln \left (\frac{1}{\err} \right )}[\Bigg] .
    \end{align*}
\end{restatable}



%
%

\section{Polymer Dynamics}\label{sec:polymerDynamics}

We analyze the following Markov chain for a polymer model with a polymer clique cover.

\begin{definition}[Polymer Clique Dynamics]
    \label{def:markov_chain}
    Let~\polymerModel be a polymer model, and let~\polymerClique be a polymer clique cover of~\polymerModel with size~\numberOfCliques. We define \markov[\polymerModel] to be a Markov chain with state space~\polyfams. Let $(X_t)_{t \in \N}$ denote a (random) sequence of states of~\markov[\polymerModel], where~$X_0$ is arbitrary.
    
    Then, for all $t \in \N$, the transitions of \markov[\polymerModel] are as follows:
    \begin{algorithmic}[1]
        \State\label{line:choosePolymerClique}choose $i \in [m]$ uniformly at random\,;
        \State\label{line:samplePolymer}choose $\polyfam \in \polyfams[][\polymerClique[i]]$ according to $\gibbsDistribution[][\polymerClique[i]]$\,;
        \State\label{line:removePolymer}\lIf{$\polyfam = \emptyset$}
        {
            $X_{t+1} = X_{t} \setminus \polymerClique[i]$
        }
        \State\label{line:addPolymer}\lElseIf{\emph{$X_{t} \cup \polyfam$ is a valid polymer family}}
        {
            $X_{t+1} = X_{t} \cup \polyfam$
        }
        \State\label{line:doNotChange}\lElse
        {
            $X_{t+1} = X_{t}$
        }
    \end{algorithmic}
\end{definition}



Given a polymer model $\polymerModel = (\polys, \weight, \ncomp)$ and a polymer Markov chain~\markov[\polymerModel], let~\transitionProbabilitiesName denote the transition matrix of~\markov[\polymerModel]. That is, for all $\polyfam, \polyfam' \in \polyfams[\polymerModel]$, the entry \transitionProbabilities{\polyfam}{\polyfam'} denotes the probability to transition from state~\polyfam to state~$\polyfam'$ in a single step. Note that~\transitionProbabilitiesName is time-homogeneous and that, for all $\polyfam, \polyfam' \in \polyfams[\polymerModel]$ with $\transitionProbabilities{\polyfam}{\polyfam'} > 0$, it holds that the symmetric difference of~\polyfam and~$\polyfam'$ has a cardinality of at most~$1$, since the polymer families of a polymer clique are all singletons. Further note that~\markov[\polymerModel] has a positive self-loop probability, as the polymers from a polymer clique are pairwise incompatible.

The transition probabilities of two neighboring states of~\markov[\polymerModel] follow a simple pattern. In order to ease notation, for all $\polymer \in \polys$, let $z_{\polymer} = \sum_{i \in [m]\colon \gamma \in \polymerClique[i]} 1/\partitionFunction[][\polymerClique[i]]$. For all $\polyfam, \polyfam' \in \polyfams[\polymerModel]$ such that there is a $\polymer \in \polys$, $\polymer \notin \polyfam$ such that $\polyfam' = \polyfam \cup \{\polymer\}$, it holds that
\begin{align*}
    \numberthis\label{eq:transitions}
    \transitionProbabilities{\polyfam}{\polyfam'}
        &= \frac{1}{m} \sum_{\substack{i \in [m]\colon\\ \polymer \in \polymerClique[i]}} \gibbsDistributionFunction{\{\polymer\}}[][\polymerClique[i]]
        = \frac{1}{m} \sum_{\substack{i \in [m]\colon\\ \polymer \in \polymerClique[i]}} \frac{\weight[\polymer]}{\partitionFunction[][\polymerClique[i]]}
        = w_{\gamma} \frac{z_{\gamma}}{m} > 0 \ \textrm{ and that}\\
    \transitionProbabilities{\polyfam'}{\polyfam}
        &= \frac{1}{m} \sum_{\substack{i \in [m]\colon\\ \polymer \in \polymerClique[i]}} \gibbsDistributionFunction{\{\emptyset\}}[][\polymerClique[i]]
        = \frac{1}{m} \sum_{\substack{i \in [m]\colon\\ \polymer \in \polymerClique[i]}} \frac{1}{\partitionFunction[][\polymerClique[i]]}
        = \frac{z_{\gamma}}{m} > 0.
\end{align*}

The polymer clique dynamics are suitable for sampling from the Gibbs distribution of a polymer model, since the limit distribution of the Markov chain converges to~\gibbsDistribution.

\begin{restatable}{lemma}{convergenceLemma}
	\label{lemma:convergence}
	Let~\polymerModel be a polymer model. The polymer Markov chain~\markov[\polymerModel] is ergodic with stationary distribution~\gibbsDistribution[\polymerModel].
\end{restatable}

Recall \Cref{def:pmc} (clique dynamics) from the introduction. Assuming that the condition holds, we obtain the following bound on the mixing time of~\markov[\polymerModel].

\begin{restatable}{lemma}{mixingLemma}
	\label{lemma:mixing}
    Let $\polymerModel = (\polys, w, \ncomp)$ be a polymer model satisfying the \pmc with function~$f$, and let \polymerClique be a polymer clique cover of~\polymerModel with size~\numberOfCliques. Then, for all $\varepsilon \in (0, 1]$, it holds that
	\[
		\mix{\markov[\polymerModel]}{\err} \in \bigO{
			\frac{m^3}{\min_{\gamma \in \polys}\{z_{\gamma}\}}
			\ln \left( m 
				\frac
					{\max_{\gamma \in \polys} \left\{ \frac{f(\gamma)}{z_{\gamma} (1+w_{\gamma})} \right\}}
					{\min_{\gamma \in \polys} \left\{ \frac{f(\gamma)}{z_{\gamma} (1+w_{\gamma})} \right\}}
				\right)^2
			\ln \left( \frac{1}{\err} \right)	
			}.
	\]
\end{restatable}

Last, we combine \Cref{lemma:convergence,lemma:mixing}, and observe that $1/\prt_{\max} \le z_{\gamma} \le m$ and $1 \le 1 + w_{\gamma} \le \prt_{\max}$ for all polymers $\gamma \in \polys$ to obtain the main result of this section.

\begin{theorem}
    \label{thm:markov_chain}
    Let $\polymerModel = (\polys, w, \ncomp)$ be a polymer model, let \polymerClique be a polymer clique cover of~\polymerModel with size~\numberOfCliques, and let $\prt_{\max} = \max_{i \in [\numberOfCliques]} \{\partitionFunction[][\polymerClique[i]]\}$. Further, assume that~\polymerModel satisfies the \pmc with function~$f$, and let $f_{\max} = \max_{\polymer \in \polys} \{f(\polymer)\}$ and $f_{\min} = \min_{\polymer \in \polys} \{f(\polymer)\}$.
    
    Then the Markov chain \markov[\polymerModel] has the unique stationary distribution \gibbsDistribution[\polymerModel] and, for all $\varepsilon \in (0, 1]$, it holds that
    \[
        \mix{\markov[\polymerModel]}{\err} 
        \in \bigO{
            m^3 \prt_{\max}
            \ln \left( 
            m^2 \prt_{\max}^2
            \frac{f_{\max}}{f_{\min}}
            \right)^2
            \ln \left( \frac{1}{\err} \right)
        }.
    \]
\end{theorem}

\subsection{Comparison to Conditions for Cluster Expansion}\label{sec:comparison}
In order to set our \pmc in the context of existing conditions for absolute convergence of the cluster expansion, we compare it to the condition of Fernández and Procacci~\cite{FP07}. We choose it for comparison because it is, to the best of our knowledge, the least restrictive condition for absolute convergence of the cluster expansion of abstract polymer models. As Fernández and Procacci show, their condition is an improvement over other known conditions, including the Dobrushin condition~\cite{Dob96} and the Koteck{\'y}--Preiss condition~\cite{1986:Kotecky:cluster_expansion_polymer_models}.

\begin{definition}[From~{\cite{FP07}}] \label{def:fpc}
	 Let $\polymerModel = (\polys, w, \ncomp)$ be a polymer model, and let $\neighbors{\polymer} = \set{\polymer' \in \polys}{\polymer' \ncomp \polymer}$.
	 We say that $\polymerModel$ \emph{satisfies the \fpc} if and only if there is a function $f\colon \polys \to \R_{> 0}$ such that, for all $\polymer \in \polys$, it holds that
	 \[
		\sum_{\polyfam \in \polyfams[\polymerModel][\neighbors{\polymer}]} \prod_{\polymer' \in \polyfam} f(\polymer') \weight[\polymer'] \le f(\polymer) .
	 \]
\end{definition}

Note that we state the condition slightly differently from the version of the original authors to ease comparison.
The original form is recovered by setting $f\colon \polymer \mapsto \hat{f}(\polymer)/\weight[\polymer]$ for some function $\hat{f}\colon \polys \to \R_{> 0}$. 
Further, the original version allows $f$ (or $\hat{f}$ respectively) to take the value~$0$.
However, note that if $f(\polymer) = 0$ for any $\polymer \in \polys$, then the condition is trivially void because $\emptyset \in \polyfams[\polymerModel][\neighbors{\polymer}]$, which lower bounds the left hand side of the inequality by $1$.

The following statement shows how our \pmc relates to the \fpc as given in \Cref{def:fpc}.
\begin{restatable}{proposition}{comparisonProposition} 
	\label{prop:comparison}
	If a polymer model $\polymerModel = (\polys, \weight, \ncomp)$ satisfies the \fpc for a function $f$, then it also satisfies the \pmc for the same function.
\end{restatable}

Note that \Cref{prop:comparison} implies that if a polymer model satisfies the \fpc for a function $f$, then \Cref{thm:markov_chain} bounds the mixing time of the polymer Markov chain for any given clique cover.
Further, \Cref{prop:comparison} and its implied mixing time bounds for the polymer Markov chain carry over to all convergence conditions that are more restrictive than the \fpc, such as the Dobrushin condition and the Koteck{\'y}--Preiss condition.

\section{Algorithmic Results}\label{sec:algo}

We now discuss how the polymer Markov chain~\markov of a polymer model~\polymerModel with a clique cover of size~\numberOfCliques is used to approximate~\partitionFunction[\polymerModel] in a randomized fashion.
To this end, \markov is turned into an approximate sampler for~\polymerModel (\Cref{thm:sampling}).
Then this sampler is applied in an algorithmic framework (\Cref{algo:appx_prt}) that yields an $\err$-approximation of~\partitionFunction[\polymerModel] (\Cref{thm:appx_partition_function}).
Under certain assumptions, such as that the restricted partition function of each polymer clique is in $\poly{\numberOfCliques}$, the approximation is computable in time $\poly{\numberOfCliques/\err}$. 

In order to discuss the computation time of operations on a polymer model rigorously, we need to make assumptions about the operations we consider and their computational cost.
To this end, we say that a polymer model $\polymerModel = (\polys, w, \ncomp)$ with a polymer clique cover~\polymerClique of size~\numberOfCliques is \emph{computationally feasible} if and only if all of the following operations can be performed in time $\poly{\numberOfCliques}$:
\begin{enumerate}[(1)]
    \item \label[assume]{assume:sampling:cover}
        for all $i \in [\numberOfCliques]$, we can draw~$\clique_i$ uniformly at random,
    \item \label[assume]{assume:sampling:clique}
        for all $i \in [\numberOfCliques]$ and all $\gamma \in \polys$, we can check whether $\gamma \in \clique_i$,
    \item \label[assume]{assume:sampling:ncomp}
        for all $\gamma, \gamma' \in \polys$, we can check whether $\gamma \ncomp \gamma'$,
    \item \label[assume]{assume:sampling:weights}
        for all $\polymer \in \polys$, we can compute~$\weight[\polymer]$.
\end{enumerate}

In addition to the more complex operations above, we further assume that, for all $\polymer \in \polys$ and all $\polyfam \in \polyfams$, we can compute $\polyfam \setminus \{\polymer\}$ and $\polyfam \cup \{\polymer\}$, and we can decide whether $\polyfam = \emptyset$ in time $\poly{\numberOfCliques}$.


\subsection{Sampling From the Gibbs Distribution}

We discuss under what assumptions one can approximately sample from the Gibbs distribution of a computationally feasible polymer model in time polynomial in the size of the clique cover. One of our main results is the following, which we recall from \Cref{sec:intro}.

\samplingTheorem*
%

By making a slightly stronger assumption about the polymer model, assumptions~(\ref{thm:sampling:prt}) and~(\ref{thm:sampling:pmc}) of \Cref{thm:sampling} are easily satisfied.

\begin{observation}\label{remark:spmc}
	If $\polymerModel$ satisfies, for all $\polymer \in \polys$,
	\begin{equation} \label{eq:spmc}
		\sum_{\polymer' \in \polys\colon \polymer' \ncomp \polymer} f(\polymer') \weight[\polymer'] \le f(\polymer) ,
	\end{equation}
	then the \pmc is satisfied for the same function $f$. 
	Thus, if \cref{eq:spmc} holds for an appropriate function $f$, assumption~(\ref{thm:sampling:pmc}) also holds.
	Further, by setting~$\polymer$ to be the polymer in $\polymerClique[i]$ that minimizes $f$, \cref{eq:spmc} implies that $\partitionFunction[][\polymerClique[i]] = 1 + \sum_{\polymer' \in \polymerClique[i]} \weight[\polymer'] \le 2$, meaning that \cref{thm:sampling:prt} is trivially satisfied. 
\end{observation}

Although the condition above is slightly more restrictive than the \pmc, it is more convenient to use for algorithmic applications. It can be seen as a weaker and more general version of the mixing condition by Chen et~al.~\cite{CGGPSV19}.

\subsection{Approximation of the Partition Function}\label{sec:approx_prt}

By now, we mainly discussed conditions for approximately sampling from the Gibbs distribution.
We now discuss how to turn this into a randomized approximation for the partition function.
To this end, we apply self-reducibility \cite{JVV86}.
However, note that the obvious way for applying self-reducibility, namely based on single polymers, might take $|\polys|$ reduction steps.
This is not feasible in many algorithmic applications of polymer models.

To circumvent this problem, we propose a self-reducibility argument based on polymer cliques.
By doing so, the number of reductions is bounded by the size of the clique cover that is used, thus adding no major overhead to the runtime of our proposed approximate sampling scheme.
Besides this idea of applying self-reducibility based on cliques, most of our arguments are analogous to known applications, like in \cite[Chapter $3$]{jerrum2003counting}.

We proceed by formalizing clique-based self-reducibility.
Let $\polymerModel = (\polys, \weight, \ncomp)$ be a polymer model, and let~$\clique$ be a polymer clique cover of~\polymerModel with size~$\numberOfCliques$.
We define a sequence of subsets of polymers $(\Uclique{i})_{0 \le i \le \numberOfCliques}$ with $\Uclique{0} = \emptyset$ and, for $i \in [\numberOfCliques]$, with $\Uclique{i} = \Uclique{i-1} \cup \polymerClique[i]$.

Further, for all $i \in [\numberOfCliques]$, let $\prtFrac[i] = \partitionFunction[][\Uclique{i-1}]/\partitionFunction[][\Uclique{i}]$.
Note that $\partitionFunction[][\Uclique{0}] = 1$ and $\partitionFunction[][\Uclique{\numberOfCliques}] = \partitionFunction$.
It holds that 
\[
	\partitionFunction 
	= \prod_{i \in [\numberOfCliques]} \frac{\partitionFunction[][\Uclique{i}]}{\partitionFunction[][\Uclique{i-1}]}
	= \left( \prod_{i \in [\numberOfCliques]} \prtFrac[i] \right)^{-1} .
\]
Hence, when approximating~\partitionFunction, it is sufficient to focus, for all $i \in [\numberOfCliques]$, on approximating $\prtFrac[i]$.

For all $i \in [\numberOfCliques]$, a similar relation holds with respect to the probability that a random $\polyfam \in \polyfams[][\Uclique{i}]$ is already in $\polyfams[][\Uclique{i-1}]$. More formally, let $i \in [\numberOfCliques]$, and let $\polyfam \sim \gibbsDistribution[][\Uclique{i}]$. Note that
\begin{equation}
    \label{eq:expectationOfIndicatorFunctionForApproximation}
	\E{\indicator{\polyfam \in \polyfams[][\Uclique{i-1}]}} 
	= \sum_{\polyfam \in \polyfams[][\Uclique{i}]} \gibbsDistributionFunction{\polyfam}[][\Uclique{i}] \cdot \indicator{\polyfam \in \polyfams[][\Uclique{i-1}]}
	= \sum_{\polyfam \in \polyfams[][\Uclique{i-1}]} \gibbsDistributionFunction{\polyfam}[][\Uclique{i}]
	= \frac{\partitionFunction[][\Uclique{i-1}]}{\partitionFunction[][\Uclique{i}]}
	= \prtFrac[i] .
\end{equation}

We use these observations in order to obtain a randomized approximation of~\partitionFunction (\Cref{algo:appx_prt}) by iteratively, for all $i \in [\numberOfCliques]$, approximating~$\prtFrac[i]$ by sampling from~$\gibbsDistribution[][\Uclique{i}]$.

\begin{algorithm}[t]
 \KwIn{polymer model $\polymerModel = (\polys, \weight, \ncomp)$, polymer clique cover of~\polymerModel with size~$\numberOfCliques$, number of samples $\numSamples \in \N_{> 0}$, sampling error $\err_\numSamples \in (0, 1]$}
 \KwOut{$\err$-approximation of~\partitionFunction[\polymerModel] according to \Cref{lemma:appx_algo}}
 \For{$i \in [\numberOfCliques]$}{
    \For{$j \in [\numSamples]$}{
        $\polyfam^{(j)} \gets \err_\numSamples$-approximate sample from $\gibbsDistribution[][\Uclique{i}]$\;
    }
    $\appxPrtFrac[i] \gets \frac{1}{\numSamples} \sum_{j \in [\numSamples]} \indicator{\polyfam^{(j)} \in \polyfams[][\Uclique{i-1}]}$\;
 }
 $\appxPrtFrac \gets \prod_{i \in [\numberOfCliques]} \appxPrtFrac[i]$\;
 \Return $1/\appxPrtFrac$\;
 \caption{Randomized approximation of the partition function of a polymer model}
 \label{algo:appx_prt}
\end{algorithm}

The following result bounds, for all $\err \in (0, 1]$, the number of samples~$\numSamples$ and the sampling error~$\err_\numSamples$ that are required by \Cref{algo:appx_prt} to obtain an $\err$-approximation of $\partitionFunction$.

\begin{restatable}{lemma}{appxAlgoLemma}
	\label{lemma:appx_algo}
	Let $\polymerModel = (\polys, w, \ncomp)$ be a polymer model, let~\polymerClique be a polymer clique cover of~\polymerModel with size~\numberOfCliques, let $\prt_{\max} = \max_{i \in [\numberOfCliques]} \{\partitionFunction[][\polymerClique[i]]\}$, and let $\err \in (0, 1]$.
	Consider \Cref{algo:appx_prt} for~\polymerModel with $\numSamples = 1 + 125 \prt_{\max} \numberOfCliques/\err^2$ and $\err_{\numSamples} = \err/(5 \prt_{\max} \numberOfCliques)$.
    Then \Cref{algo:appx_prt} returns a randomized $\err$-approximation of $\partitionFunction$.
\end{restatable}

Based on \Cref{algo:appx_prt} and \Cref{lemma:appx_algo}, we now recall our main theorem on the approximation of the partition function of an abstract polymer model.

\approximationPartitionFunction*
%

%

\section{Truncation of Polymer Cliques}\label{sec:truncation}

In \Cref{sec:algo}, we discuss under which assumptions the partition function of a polymer model~\polymerModel with polymer clique cover~\polymerClique of size~\numberOfCliques can be approximated in time polynomial in~\numberOfCliques (\Cref{thm:appx_partition_function}). One of the assumptions requires to be able to sample, for all $i \in [\numberOfCliques]$, from~\gibbsDistribution[][\polymerClique[i]] in time \poly{\numberOfCliques}.
Unfortunately, for many algorithmic problems, the number of polymer families of each polymer clique is large, and efficient sampling from~\gibbsDistribution[][\polymerClique[i]] is non-trivial.
However, as we only require to approximately sample from~\gibbsDistribution[][\polymerClique[i]], it is sufficient to ignore polymer families with low probabilities, that is, with low weight.

We formalize this concept rigorously by defining a size function for polymers.
We aim to remove polymers of large size (low weight), which still yields a sufficient approximation of~\gibbsDistribution[][\polymerClique[i]] (\Cref{lemma:clique_truncation}).
As a consequence, we can still approximate~\partitionFunction[\polymerModel] in time polynomial in~\numberOfCliques (\Cref{thm:sampling_trunc}).

\begin{definition}[Size Function]
    \label{def:size}
    Given a polymer model $(\polys, w, \ncomp)$, a size function is a function $\size{\cdot}: \polys \to \R_{>0}$.
    For a fixed size function $\size{\cdot}$ and some polymer $\gamma \in \polys$, we call $\size{\gamma}$ the \emph{size} of $\gamma$.
\end{definition}

Given a size function, we \emph{truncate} the polymer model to polymers of small size.

\begin{definition}[Truncation]
    \label{def:truncation}
    Let $(\polys, w, \ncomp)$ be a polymer model equipped with a size function~$\size{\cdot}$, and let $\mathcal{B} \subseteq \polys$.
    For all $k \in \R$, we call $\mathcal{B}^{\le k} = \{\gamma \in \mathcal{B} \mid \size{\gamma} \le k\}$ the \emph{truncation} of $\mathcal{B}$ to size $k$.
    Further, we write $\mathcal{B}^{>k} = \mathcal{B} \setminus \mathcal{B}^{\le k}$.
\end{definition}

Note that $\mathcal{B} \subseteq \polys$ and that $\mathcal{B}^{\le k}, \mathcal{B}^{> k}$ is a partitioning of $\mathcal{B}$, which implies $\mathcal{B}^{\le k}, \mathcal{B}^{> k} \subseteq \polys$.
Thus, we can apply our notions of restricted polymer families, partition function, and Gibbs distribution as stated in \Cref{sec:restrictedPolymerModels} to $\mathcal{B}^{\le k}$ and $\mathcal{B}^{> k}$ as well.
The case $\mathcal{B} = \polys$ (i.e., we truncate the entire polymer model) plays a special role, which is why we use the shorter notation $\polyfams_{\le k} = \polyfams_{\polys^{\le k}}$, $\prt_{\le k} = \prt_{\polys^{\le k}}$, and $\gibbs_{\le k} = \gibbs_{\polys^{\le k}}$.
Analogously, we define $\polyfams_{> k}$, $\prt_{> k}$, and~$\gibbs_{> k}$.

\begin{restatable}{lemma}{cliqueTruncationLemma}
	\label{lemma:clique_truncation}
	Let $\polymerModel = (\polys, w, \ncomp)$ be a polymer model, let~\polymerClique be a polymer clique cover of~\polymerModel with size~\numberOfCliques, and let $\size{\cdot}$ be a size function for~\polymerModel.
	Assume that there is a $k \in \R$ and an $\err \in (0, 1)$ such that, for all $i \in [\numberOfCliques]$, it holds that
	\begin{align}
        \label{eq:truncationLemmaAssumption}
		\sum_{\gamma \in \clique_i^{>k}} w_{\gamma} \le \frac{\err}{m}.
	\end{align}
    
	Then $\eulerE^{- \err} \le \prt_{\le k}/\prt \le 1$ and $\dtv{\gibbs}{\gibbs_{\le k}} \le \err$.
\end{restatable}



Recall the \ctc (\Cref{def:ctc}). If the \ctc is satisfied, by choosing a reasonable value~$k$ for truncating a polymer model, only little overall weight is removed.
That is, the truncated model represents a good approximation of the original.

\begin{restatable}{lemma}{ctcLemma}
    \label{lemma:ctc}
    Let $\polymerModel = (\polys, w, \ncomp)$ be a polymer model, let~\polymerClique be a polymer clique cover of~\polymerModel with size~\numberOfCliques, let $\size{\cdot}$ be a size function for~\polymerModel, and let $i \in [\numberOfCliques]$.
    Assume that~$\clique_i$ satisfies the \ctc{} for a function~$g$ and a bound~$B$.
    
    Then, for all $\err' \in (0, 1)$ and all $k \ge g^{-1} \left( B/\err' \right)$, it holds that
    \[
    	\sum_{\gamma \in \clique_i^{>k}} w_{\gamma} \le \err' .
    \]
\end{restatable}

As a direct consequence of \Cref{lemma:ctc}, we get that the partition function of the truncated model is a useful approximation of the original partition function.
Combining \Cref{lemma:clique_truncation,lemma:ctc} and choosing $\err' = \err/m$ directly implies the following result.

\begin{corollary}
    \label{lemma:model_truncation}
    Let $\polymerModel = (\polys, w, \ncomp)$ be a polymer model, let~\polymerClique be a polymer clique cover of~\polymerModel with size~\numberOfCliques, and let $\size{\cdot}$ be a size function for~\polymerModel.
    Assume that there is a $g\colon \R \to \R_{>0}$ and a $B \in \R_{> 0}$ such that, for $i \in [\numberOfCliques]$, the polymer clique~$\polymerClique[i]$ satisfies the \ctc{} for~$g$ and~$B$.
    
    Then, for all $\err \in (0, 1)$ and all $k \ge g^{-1} \left( B m/\err \right)$, it holds that $\eulerE^{- \err} \le \prt_{\le k}/\prt \le 1$ and $\dtv{\gibbs}{\gibbs_{\le k}} \le \err$.
\end{corollary}

Using the truncated polymer model, we achieve an $\err$-approximation result of the partition function of the original model that is computable in time $\poly{\numberOfCliques/\err}$, similar to \Cref{thm:appx_partition_function}.

\begin{restatable}{theorem}{samplingTruncTheorem}
	\label{thm:sampling_trunc}
    Let $\polymerModel = (\polys, w, \ncomp)$ be a computationally feasible polymer model, let~\polymerClique be a polymer clique cover of~\polymerModel with size~\numberOfCliques, and let $\size{\cdot}$ be a size function for~\polymerModel.
    Further, let $\prt_{\max} = \max_{i \in [\numberOfCliques]} \{\partitionFunction[][\clique_i]\}$, and let $t(k)$ denote an upper bound, for all $i \in [\numberOfCliques]$, on the time to enumerate~$\clique_i^{\le k}$.
    Last, assume that
    \begin{enumerate}[(a)]
        \item \label[assume]{thm:samplingTrunc:prt}
            $\prt_{\max} \in \poly{\numberOfCliques}$,
        \item \label[assume]{thm:samplingTrunc:pmc}
            \polymerModel satisfies the \pmc for a function $f$ such that, for all $\polymer \in \polys$, it holds that $\eulerE^{-\poly{\numberOfCliques}} \le f(\gamma) \le \eulerE^{\poly{\numberOfCliques}}$, and that
        \item \label[assume]{thm:samplingTrunc:ctc}
            there are $g\colon \R \to \R_{> 0}$ and $B \in \R_{> 0}$ with $B \in \poly{\numberOfCliques}$ and $t(g^{-1}(x)) \in \poly{x}$ (for all $x \in \R_{> 0}$) such that, for all $i \in [\numberOfCliques]$, it holds that~$\clique_i$ satisfies the \ctc{}.
    \end{enumerate}

    Then, for all $\err \in (0, 1]$, we can $\err$-approximately sample from~$\gibbs$ in time $\poly{\numberOfCliques/\err}$, and there is a randomized $\err$-approximation of~$\prt$ computable in $\poly{\numberOfCliques/\err}$.
\end{restatable}

Note that \Cref{remark:spmc} applies to \Cref{thm:sampling_trunc} as well.
That is, by using more restrictive assumptions, assumptions~(\ref{thm:samplingTrunc:prt}) and~(\ref{thm:samplingTrunc:pmc}) are satisfied.

The results from \Cref{table:bounds} were obtained by using \Cref{thm:sampling_trunc} together with \Cref{remark:spmc}.
To demonstrate how the bounds are calculated, we showcase this for the hard-core model on bipartite expanders in \Cref{appendix:expanders}. 
Moreover, we discuss there how to choose the function~$f$ for the \pmc and describe how the other results in \Cref{table:bounds} are obtained.

\section*{Acknowledgments}

This work was supported by the Paris Île-de-France Region via the European Union’s Horizon 2020 research and innovation program under the Marie Skłodowska-Curie grant agreement No. 945298-ParisRegionFP.

\newpage
\bibliography{references.bib}

\newpage
\appendix

\section{\namedAppendix{Coupling Lemma}} \label{appendix:coupling_lemma}
We discuss \Cref{lemma:exp_potential} in detail. First, we explain why the assumptions of the original theorem by Greenberg et~al.~\cite[Theorem~$3.3$]{GreenbergPR09IncorrectSupermartingaleDrift} are insufficient. With \Cref{ex:counterexampleToCouplingTheorem}, we provide a counterexample. Last, we prove our version of the theorem.


Besides some minor generalizations, the most important difference between \Cref{lemma:exp_potential} and Theorem~$3.3$ by Greenberg et~al.~\cite{GreenbergPR09IncorrectSupermartingaleDrift} is that we assume the coupling to be defined for \emph{all} pairs of states.
We also require the expected change of~$\delta$ as well as the probability bound to hold for \emph{all} pairs of states.
In contrast, Greenberg et~al.~\cite{GreenbergPR09IncorrectSupermartingaleDrift} claim that it is sufficient if these properties hold for neighboring states with respect to some adjacency structure.
In what follows, we argue that this does not always suffice.

It is well known that couplings on adjacent states can be extended to all pairs of states such that the expected decrease of $\delta$ for adjacent states implies an expected decrease for all pairs of states~\cite{DyerG98MultiplicativeCoupling}.
However, a similar argument does not necessarily hold for bounds on the probability that $\delta$ changes by at least a certain amount.
More precisely, it is possible to construct a Markov chain and a coupling such that
\[
\Pr{|\delta(X_{t+1}, Y_{t+1}) - \delta(x, y)| \ge \eta \delta(x, y)}[X_t = x, Y_t = y] \ge \kappa
\] 
holds for all pairs of adjacent states $x, y \in \Omega$ but not for all pairs of non-adjacent states.

Thus, Theorem~$3.3$ by Greenberg et~al.~\cite{GreenbergPR09IncorrectSupermartingaleDrift} can be used to deduce upper bounds for mixing times that contradict known lower bounds.
We demonstrate this by giving a simple counterexample (\Cref{ex:counterexampleToCouplingTheorem}).
Using Theorem~$3.3$ by Greenberg et~al.~\cite{GreenbergPR09IncorrectSupermartingaleDrift}, we bound the mixing time of a symmetric random walk on a cycle of size~$n$ by $\bigO{\ln (n)^2 \ln \left( 1/\err \right)}$.
This contradicts the lower bound of $\bigOmega{n \ln \left( 1/\err \right)}$ that results from the diameter of the state space \cite[Chapter $7.1.2$]{levin2017markov}.

\begin{example}
	\label{ex:counterexampleToCouplingTheorem}
	We consider a symmetric random walk on a cycle of length $n \in \N_{> 2}$ (i.e., $\Omega = \{0\} \cup [n-1]$).
	In what follows, let all $+1$ and $-1$ operations on the state space be defined modulo~$n$.
	In order to have the desired self-loop probability, we define the transitions~$P$, for all $x \in \Omega$, by $P(x, x) = 1/2$ and $P(x, x+1) = P(x, x-1) = 1/4$.
	
	We say two states $x, y \in \Omega$ with $x \neq y$ are adjacent if and only if $x = y + 1$ or $x = y - 1$.
	Further, we define $\delta$ to be the shortest-path distance in the cycle.
	Note that, for all $x, y \in \Omega$ with $x \neq y$, it holds that $\delta(x, y) \in \left[ 1, \left\lfloor n/2 \right\rfloor \right]$.
	
	Let $(X_t)_{t \in \N}$ and $(Y_t)_{t \in \N}$ be two copies of the chain $(\Omega, P)$, and let $x, y \in \Omega$ be adjacent.
	Without loss of	generality, assume $x = y + 1$.
	For $X_t = x, Y_t = y$ we construct the following coupling:
	\begin{itemize}
		\item With probability $1/4$, choose $X_{t+1} = x$ and $Y_{t+1} = x$, resulting in $\delta(X_{t+1}, Y_{t+1}) = 0$.
		\item With probability $1/4$, choose $X_{t+1} = y$ and $Y_{t+1} = y$, resulting in $\delta(X_{t+1}, Y_{t+1}) = 0$.
		\item With probability $1/4$, choose $X_{t+1} = x$ and $Y_{t+1} = y$, resulting in $\delta(X_{t+1}, Y_{t+1}) = 1$.
		\item With the remaining probability of $1/4$, choose $X_{t+1} = x + 1$ and $Y_{t+1} = y - 1$, resulting in $\delta(X_{t+1}, Y_{t+1}) = 3$ .
	\end{itemize}
	Note that $\E{\delta(X_{t+1}, Y_{t+1})}[X_t = x, Y_t = y] = \delta(x, y)$.
	
	For $\eta = 0.999$ and $\kappa = 3/4$, it holds that
	\[
	\Pr{|\delta(X_{t+1}, Y_{t+1}) - \delta(x, y)| \ge \eta \delta(x, y)}[X_t = x, Y_t = y] \ge \kappa .	
	\]
	Theorem~$3.3$ by Greenberg et~al.~\cite{GreenbergPR09IncorrectSupermartingaleDrift} then yields a mixing time bound of $\bigO{\ln(n)^2 \ln(1/\err)}$, which contradicts the linear lower bound stated by Levin and Peres~\cite[Chapter $7.1.2$]{levin2017markov}.
\end{example}

Note that \Cref{ex:counterexampleToCouplingTheorem} is not a counterexample for \Cref{lemma:exp_potential}, as there are, for all $\eta \in \smallOmega{1/n}$, non-adjacent states $x, y \in \Omega$ with
\[
\Pr{|\delta(X_{t+1}, Y_{t+1}) - \delta(x, y)| \ge \eta \delta(x, y)}[X_t = x, Y_t = y] = 0.
\]

\subsection*{Our Version of the Theorem}

We closely follow the proof of Greenberg et~al.~\cite{GreenbergPR09IncorrectSupermartingaleDrift}. Central to this is the following theorem, which we present in a slightly different fashion than Greenberg et~al.~\cite[Lemma 3.5]{greenberg2017sampling}.

\begin{theorem}
	\label{lemma:greenberg_stopping}
	Let $d, D \in \R$ with $d \leq D$, let $(S_t)_{t \in \N}$ be a stochastic process adapted to a filtration $(\filtration[t])_{t \in \N}$, and let~$T$ be a stopping time with respect to~$\filtration$.
	Assume that, for all $t \in \N$, it holds that $S_{\min \{t, T\}} \in [d, D]$, that
	\begin{align}
		\label{eq:multiplicativeDriftCondition}
		\E{S_{t+1}}[\filtration[t]] \cdot\indicator{t < T} \le S_t \cdot\indicator{t < T},
	\end{align}
	and that there is a $Q \in \R_{> 0}$ such that
	\begin{align}
		\label{eq:varianceBound}
		\E{(S_{t+1} - S_t)^2}[\filtration[t]] \cdot\indicator{t < T} \ge Q \cdot\indicator{t < T}.
	\end{align}
	
	Then
	\begin{align*}
		\E{T}[\filtration[0]] &\le \frac{\E{(D - S_T)^2}[\filtration[0]] - (D - S_0)^2}{Q} .
	\end{align*}
\end{theorem}

Before we prove the theorem, we discuss the changes we made in the phrasing of the theorem.
Different to the original theorem by Greenberg et~al.~\cite[Lemma 3.5]{greenberg2017sampling}, we include a filtration and we use indicator functions. Our reasons are as follows. The proof of \Cref{lemma:greenberg_stopping} aims to apply the optional-stopping theorem for submartingales. A submartingale is, by definition, a stochastic process $(Z_t)_{t \in \N}$ adapted to a filtration $(\filtration[t])_{t \in \N}$ such that, for all $t \in \N$, the expectation of~$Z_t$ is finite and $\E{Z_{t + 1}}[\filtration[t]] \geq Z_t$. It is important to note that the expectation $\E{Z_{t + 1}}[\filtration[t]]$ is itself a random variable and that the inequality $\E{Z_{t + 1}}[\filtration[t]] \geq Z_t$ is stronger than $\E{Z_{t + 1}} \geq \E{Z_t}$ (which follows by the law of total expectation). Hence, we require a filtration.

Second, the indicator functions make sure that \cref{eq:multiplicativeDriftCondition,eq:varianceBound} (and the boundedness of~$S$) only have to hold as long as~$S$ did not stop. Afterward, they are trivially satisfied. This is important, as~$S$ is bounded from below by~$d$ and its conditional expectation does not increase. Assume that we did not use indicator functions. If there is a $t \in \N$ such that $S_t = d$, then $S_{t + 1} = d$ holds as well, as otherwise the inequality $\E{S_{t + 1}}[\filtration[t]] \leq S_t$ does not hold. However, this implies that $\E{(S_{t + 1} - S_t)^2}[\filtration[t]] = 0$ (since the process is now almost surely deterministic), which violates \cref{eq:varianceBound} if not for the indicator functions.

Note that our additional assumptions in \Cref{lemma:greenberg_stopping} only fix issues in the proof of Greenberg et~al.~\cite[Lemma 3.5]{greenberg2017sampling}. The proof itself remains mostly unchanged.

\begin{proof}[Proof of \Cref{lemma:greenberg_stopping}]
    Let $M \in \R_{\geq 0}$ such that, for all $t \in \N$, it holds that $M \geq (D - S_t)^2 \cdot \indicator{t \leq T}$.
    Note that such an~$M$ exists, as, for all $t \in \N$, by assumption, $S_{\min \{t, T\}}$ is bounded.
    For all $t \in \N$, let $Z_t = \big((D - S_t)^2 - Qt - M\big) \cdot \indicator{t \leq T}$.
    Note that, for all $t \in \N$, due to the definition of~$M$ and due to $\indicator{t < T} \leq \indicator{t \leq T}$, it holds that
    \begin{equation}
        \label{eq:greenberg_stopping_final_step}
        \big((D - S_t)^2 - Qt - M\big) \cdot \indicator{t < T} \geq \big((D - S_t)^2 - Qt - M\big) \cdot \indicator{t \leq T} = Z_t.
    \end{equation}
    
    We show that~$Z$ is a submartingale with respect to~$\filtration$.
    Let $t \in \N$.
    By noting that $\indicator{t + 1 \leq T} = \indicator{t < T}$, since~$T$ is integer, and by applying \cref{eq:multiplicativeDriftCondition,eq:varianceBound,eq:greenberg_stopping_final_step}, we get
    \begin{align*}
        \E{Z_{t + 1}}[\filtration[t]] &= \E{\big((D - S_{t + 1})^2 - Q(t + 1) - M\big) \cdot \indicator{t + 1 \leq T}}[\filtration[t]]\\
        &= \E{\big((D - S_{t + 1})^2 - Q(t + 1) - M\big) \cdot \indicator{t < T}}[\filtration[t]]\\
        &= \big(D^2 - 2D\E{S_{t + 1}}[\filtration[t]] + \E{S_{t + 1}^2}[\filtration[t]] - Q - Qt - M\big) \cdot \indicator{t < T}\\
        &\overset{\clap{\scriptsize\eqref{eq:varianceBound}}}{\geq} \big(D^2 - 2D\E{S_{t + 1}}[\filtration[t]] + 2S_t\E{S_{t + 1}}[\filtration[t]] - S_t^2 - Qt - M\big) \cdot \indicator{t < T}\\
        &= \big(D^2 - 2\E{S_{t + 1}}[\filtration[t]](D - S_t) - S_t^2 - Qt - M\big) \cdot \indicator{t < T}\\
        &\overset{\clap{\scriptsize\eqref{eq:multiplicativeDriftCondition}}}{\geq} \big(D^2 - 2S_t(D - S_t) - S_t^2 - Qt - M\big) \cdot \indicator{t < T}\\
        &= \big((D - S_t)^2 - Qt - M\big) \cdot \indicator{t < T}\\
        &\overset{\clap{\scriptsize\eqref{eq:greenberg_stopping_final_step}}}{\geq} Z_t.
    \end{align*}
    
    Since~$Z$ is bounded, it is uniformly integrable~\cite[Theorems~$4.2.11$ and~$4.6.4$]{Durrett19ProbabilityTheoryExamples}.
    By applying the optional-stopping theorem for uniformly integrable submartingales~\cite[Theorem~$4.8.3$]{Durrett19ProbabilityTheoryExamples}, we get that $\E{Z_T}[\filtration[0]] \geq \E{Z_0}[\filtration[0]]$.
    Solving this inequality for $\E{T}[\filtration[0]]$ and noting that~$Z_0$ is $\filtration[0]$-measurable concludes the proof.
\end{proof}

While we state \Cref{lemma:greenberg_stopping} in an elaborate fashion, we use it in a different way in the following proof of \Cref{lemma:exp_potential}.
Instead of considering expected values conditional on a $\upsigma$-algebra, we consider expected values conditional on each of the outcomes of the random process before it stops.
We formalize this statement in the following remark.
\begin{remark}
    \label{rem:filtrations_to_events}
    Let $(X_t)_{t \in \N}$ denote a random process over~$\R$, defined over a probability space with sample space~$\Omega$, let~$(\filtration[t])_{t \in \N}$ denote the natural filtration of~$X$, and let~$T$ be a stopping time with respect to~$\filtration$.
    If~$X$ is discrete and Markovian, then, for all $t \in \N$ and all $\delta \in \R$, the following equivalence holds:
    \begin{align}
        \label[prop]{eq:equivalence_first_line}
        &\forall x \in \{y \in \R \mid \exists \omega \in \Omega\colon X_t(\omega) = y \land t < T(\omega)\}\colon \E{X_t - X_{t + 1}}[X_t = x] \leq \delta \\
        \nonumber
        &\leftrightarrow\\
        \label[prop]{eq:equivalence_second_line}
        &\E{(X_t - X_{t + 1}) \cdot \indicator{t < T}}[\filtration[t]] \leq \delta \cdot \indicator{t < T}.
    \end{align}
\end{remark}

Informally, on the one hand, \cref{eq:equivalence_second_line} considers all instantiations $\omega \in \Omega$ of~$X$ where~$X$ did not stop up to time~$t$.
On the other hand, \cref{eq:equivalence_first_line} partitions~$\Omega$ with respect to~$X_t$ and shows the relevant inequality for the sets $X_t^{-1}(x)$ (for the values of~$x$ as specified in the proposition).
Since for all $\omega \in X_t^{-1}(x)$ the value $\E{X_t - X_{t + 1}}[X_t](\omega)$ is the same (since~$X$ is Markovian), it suffices to consider their average, which is $\E{X_t - X_{t + 1}}[X_t = x]$.

More formally, first note that the indicator functions in \cref{eq:equivalence_second_line} relate to the predicate ``$t < T(\omega)$'' in the set of values for~$x$ in \cref{eq:equivalence_first_line}, which makes sure to consider only such values~$x$ for which the process did not stop yet, as these are the only instantiations of~$X$ for which \cref{eq:equivalence_second_line} is not trivially satisfied.

By definition of the conditional expectation for discrete random variables, it holds that
\begin{align*}
    \E{X_t - X_{t + 1}}[X_t = x] &= \sum_{\omega \in \Omega} \E{X_t - X_{t + 1}}[X_t](\omega) \cdot \Pr{\{\omega\}}[X_t^{-1}(x)]\\
    &=  \sum_{\substack{\omega \in \Omega\colon\\ X_t(\omega) = x}} \E{X_t - X_{t + 1}}[X_t](\omega) \cdot \frac{\Pr{\{\omega\}}}{\Pr{X_t^{-1}(x)}}.
\end{align*}
Since~$X$ is Markovian, for all $\omega \in \Omega$ where $X_t(\omega) = x$, the random variable $\E{X_t - X_{t + 1}}[X_t]$ has the same value $v \in \R$, as the distribution of~$X_{t + 1}$ only depends on $X_t = x$.
That is, $\E{X_t - X_{t + 1}}[X_t = x] = \E{X_t - X_{t + 1}}[X_t](\omega) = v$.
Since~$X$ is Markovian and thus conditioning on~$X_t$ is the same as conditioning on~$\filtration[t]$, this shows that $\E{X_t - X_{t + 1}}[X_t = x] = \E{X_t - X_{t + 1}}[\filtration[t]](\omega)$.

\paragraph*{\proofOf{\Cref{lemma:exp_potential}}}
We now start with the stating the proof of the coupling theorem, which we restate here for convenience.

\expPotentialTheorem*
\begin{proof}
	We aim to bound the expected time until~$\delta$ hits~$0$ for the coupled copies $(X_t)_{t \in \N}$ and $(Y_t)_{t \in \N}$ of~$\markov$ and for all pairs of starting states $x, y \in \Omega$.
	This results in a bound on the expected coupling time, and, because~$\markov$ is ergodic, also bounds~$\mixFunction{\markov}$ (see, for example, Chapter~$11$ by Mitzenmacher and Upfal~\cite{PAC} for a detailed discussion).
	
	We start by defining a scaled potential~$\delta'$ such that, for all $x, y \in \Omega$, it holds that $\delta'(x, y) = \delta(x, y)/d$.
	Note that~$\delta'$ takes values in $\{0\} \cup [1, D/d]$, and, for all $t \in \N$, it holds that
	\[
		X_t = Y_t \leftrightarrow \delta(X_t, Y_t) = 0 \leftrightarrow \delta'(X_t, Y_t) = 0 .
	\]
	Further, for all $x, y \in \Omega$, by the linearity of expectation and by \cref{eq:driftCondition}, it holds that
	\begin{align*}
		\E{\delta'(X_{t+1}, Y_{t+1})}[X_t=x, Y_t=y] &= \frac{1}{d} \E{\delta(X_{t+1}, Y_{t+1})}[X_t=x, Y_t=y] \\
		& \le \frac{1}{d} \delta(x, y) = \delta'(x, y) 
	\end{align*}
	and, by \cref{eq:largeJumps}, that
	\begin{align*}
		& \Pr{|\delta'(X_{t+1}, Y_{t+1}) - \delta'(x, y)| \ge \eta \delta'(x, y)}[X_t=x, Y_t=y] \\
		&\hspace{1cm}= \Pr{|\delta(X_{t+1}, Y_{t+1}) - \delta(x, y)| \ge \eta \delta(x, y)}[X_t=x, Y_t=y] \\
		&\hspace{1cm}\ge \kappa .
	\end{align*}
	
	We define the stochastic processes whose expected hitting time we bound by $\varphi = (\varphi_t)_{t \in \N}$, where $\varphi_t = \delta'(X_t, Y_t)$.
	Further, for all $x \in [0, D/d]$, let
	\begin{align*}
		\transform[x] = 
		\begin{cases}
			2\ln(2) x - 2\ln(2) &\text{ if } x \in [0, 1), \\
			\ln(x) &\text{ if } x \in [1, D/d],
		\end{cases} 
	\end{align*}
	and let $\psi = (\psi_t)_{t \in \N}$ with $\psi_t = \transform[\varphi_t]$ for all $t \in \N$.
	Note that 
	\[
		\psi_t = - 2\ln(2) \leftrightarrow \varphi_t = 0 \leftrightarrow X_t = Y_t .
	\]
	Let $T = \inf_{t \in \N} \left\{ \psi_t \leq - 2\ln(2) \right\}$.
	To obtain the desired bound on the mixing time, we bound $\E{T}[X_0 = x, Y_0 = y]$ for every pair $x, y \in \Omega$.
	
	To this end, we aim to apply \Cref{lemma:greenberg_stopping} to~$\psi$ (with the natural filtration of $(X_t, Y_t)_{t \in \N}$) together with \Cref{rem:filtrations_to_events}, which requires showing, for all $t \in \N$, all $x \in \rng{X_t}$, and all $y \in \rng{Y_t}$ such that there is an instantiation~$\omega$ from the sample space with $X_t(\omega) = x$, $Y_t(\omega) = y$, and $t < T(\omega)$, that $\E{\psi_{t+1}}[X_t = x, Y_t = y] \le \psi_{t}$ (\cref{eq:multiplicativeDriftCondition}) and obtaining a lower bound on $\E{(\psi_{t+1} - \psi_{t})^2}[X_t = x, Y_t = y]$ (\cref{eq:varianceBound}).
	
	Let $t \in \N$, $x \in \rng{X_t}$, and $y \in \rng{Y_t}$, all as described above.
    Note that this implies $\psi_{t} \in [0, \ln(D/d)]$ and $X_t \neq Y_t$.
    Further note that~$\transform$ is a concave function and that $\varphi_t \geq 1$.
	Applying Jensen's inequality to the conditional expectation, we obtain 
	\begin{align*}
		\E{\psi_{t+1}}[X_t = x, Y_t = y] 
		\le \transform[\E{\varphi_{t+1}}[X_t = x, Y_t = y]] 
		\le \transform[\varphi_{t}]
		= \psi_{t},
	\end{align*}
	which shows \cref{eq:multiplicativeDriftCondition}.
	
	We proceed by bounding $\E{(\psi_{t+1} - \psi_{t})^2}[X_t = x, Y_t = y]$ from below.
	Let~$A$ be the event that~$\psi$ jumps from $\psi_t \ge 0$ directly to $\psi_{t+1} = -2 \ln(2)$ (i.e., $\varphi_{t} \ge 1$ and $\varphi_{t+1} = 0$).
	The positive self-loop probability of~$\markov$ implies that $\Pr{A}[X_t = x, Y_t = y] < 1$ and $\Pr{\overline{A}}[X_t = x, Y_t = y][\big] > 0$.
	By the law of total expectation, we obtain
	\begin{align*}
		&\E{(\psi_{t+1} - \psi_{t})^2}[X_t = x, Y_t = y]\\ 
		&\hspace{1cm}= \E{(\psi_{t+1} - \psi_{t})^2}[X_t = x, Y_t = y, A] \Pr{A}[X_t = x, Y_t = y] \\
		&\hspace{1cm}\quad+ \E{(\psi_{t+1} - \psi_{t})^2}[X_t = x, Y_t = y, \overline{A}] (1-\Pr{A}[X_t = x, Y_t = y]) .
	\end{align*}
	
	We lower-bound each term in the sum separately.
	Because of $\psi_t \ge 0$, we have
	\begin{align} 
		\notag
		&\E{(\psi_{t+1} - \psi_{t})^2}[X_t = x, Y_t = y, A] \Pr{A}[X_t = x, Y_t = y] \\
		\notag
		&\hspace{1cm}= \big(-2 \ln(2) - \psi_{t}\big)^2 \cdot \Pr{A}[X_t = x, Y_t = y]  \\
		\label{eq:exp_potential:1}
		&\hspace{1cm}\ge 4 \ln(2)^2 \cdot \Pr{A}[X_t = x, Y_t = y] .
	\end{align}
	Furthermore, because $\eta > 0$, by conditional Markov's inequality, we get
	\begin{align*}
		&\E{(\psi_{t+1} - \psi_{t})^2}[X_t = x, Y_t = y, \overline{A}]\\
		&\hspace{1cm}\ge \ln(1+\eta)^2 \Pr{(\psi_{t+1} - \psi_{t})^2 \ge \ln(1+\eta)^2}[X_t = x, Y_t = y, \overline{A}]\\
		&\hspace{1cm}=  \ln(1+\eta)^2 \Pr{|\psi_{t+1} - \psi_{t}| \ge \ln(1+\eta)}[X_t = x, Y_t = y, \overline{A}].
	\end{align*}
	We decompose the probability as
	\begin{align*}
		&\Pr{|\psi_{t+1} - \psi_{t}| \ge \ln(1+\eta)}[X_t = x, Y_t = y, \overline{A}] \\
		&\hspace{1cm}= \Pr{\psi_{t+1} - \psi_{t} \ge \ln(1+\eta)}[X_t = x, Y_t = y, \overline{A}] \\
		&\hspace{1cm} \quad+ \Pr{\psi_{t+1} - \psi_{t} \le -\ln(1+\eta)}[X_t = x, Y_t = y, \overline{A}].
	\end{align*}
	We rewrite the first of these probabilities as
	\begin{align}
		\notag
		&\Pr{\psi_{t+1} - \psi_{t} \ge \ln(1+\eta)}[X_t = x, Y_t = y, \overline{A}] \\
		\notag
		&\hspace{1cm}= \Pr{\ln \left( \frac{\varphi_{t+1}}{\varphi_{t}} \right) \ge \ln(1+\eta)}[X_t = x, Y_t = y, \overline{A}] \\
		\notag
		&\hspace{1cm}= \Pr{\frac{\varphi_{t+1}}{\varphi_{t}} \ge 1+\eta}[X_t = x, Y_t = y, \overline{A}] \\
		\label{eq:exp_potential:2}
		&\hspace{1cm}= \Pr{\varphi_{t+1} - \varphi_{t} \ge \eta \varphi_{t}}[X_t = x, Y_t = y, \overline{A}].
	\end{align}
	Since, for all $z \in (0, 1)$, it holds that $-\ln(1+z) \ge \ln(1-z)$, we bound the other probability~by
	\begin{align}
		\notag
		&\Pr{\psi_{t+1} - \psi_{t} \le -\ln(1+\eta)}[X_t = x, Y_t = y, \overline{A}] \\
		\notag
		&\hspace{1cm}\ge \Pr{\psi_{t+1} - \psi_{t} \le \ln(1-\eta)}[X_t = x, Y_t = y, \overline{A}] \\
		\notag
		&\hspace{1cm}= \Pr{\ln \left( \frac{\varphi_{t+1}}{\varphi_{t}} \right) \le \ln(1-\eta)}[X_t = x, Y_t = y, \overline{A}] \\
		\notag
		&\hspace{1cm}= \Pr{\frac{\varphi_{t+1}}{\varphi_{t}} \le 1-\eta}[X_t = x, Y_t = y, \overline{A}] \\
		\label{eq:exp_potential:3}
		&\hspace{1cm}= \Pr{\varphi_{t+1} - \varphi_{t} \le -\eta \varphi_{t}}[X_t = x, Y_t = y, \overline{A}].
	\end{align}
	Combining \cref{eq:exp_potential:2,eq:exp_potential:3}, we obtain
	\begin{align}
		\notag
		&\Pr{|\psi_{t+1} - \psi_{t}| \ge \ln(1+\eta)}[X_t = x, Y_t = y, \overline{A}] \\
		\label{eq:exp_potential:4}
		&\hspace{1cm}\ge \Pr{|\varphi_{t+1} - \varphi_{t}| \ge \eta \varphi_{t}}[X_t = x, Y_t = y, \overline{A}] .
	\end{align}
	
	For bounding the right-hand side of \cref{eq:exp_potential:4}, recall that we consider an instantiation of the process such that $\varphi_{t} \ge 1$.
	Consider the probability that~$\varphi$ takes steps of at least size~$\eta \varphi_{t}$.
	By the law of total probability,
	\begin{align*}
		&\Pr{|\varphi_{t+1} - \varphi_{t}| \ge \eta \varphi_{t}}[X_t = x, Y_t = y] \\
		&\hspace{1cm}= \Pr{|\varphi_{t+1} - \varphi_{t}| \ge \eta \varphi_{t}}[X_t = x, Y_t = y, A] \Pr{A}[X_t = x, Y_t = y] \\
		&\hspace{1cm}\quad+ \Pr{|\varphi_{t+1} - \varphi_{t}| \ge \eta \varphi_{t}}[X_t = x, Y_t = y, \overline{A}] (1-\Pr{A}[X_t = x, Y_t = y]) . 
	\end{align*}
	Since~$A$ is the event to go from $\varphi_{t} \ge 1$ to $\varphi_{t+1} = 0$, for all $\eta \in (0, 1)$, it holds that
	\[
		\Pr{|\varphi_{t+1} - \varphi_{t}| \ge \eta \varphi_{t}}[X_t = x, Y_t = y, A] = 1 .
	\]
	Thus, and by \cref{eq:largeJumps}, we obtain
	\begin{align}
		\notag
		&\Pr{|\varphi_{t+1} - \varphi_{t}| \ge \eta \varphi_{t}}[X_t = x, Y_t = y, \overline{A}] \\
		\notag
		&\hspace{1cm}= \frac{\Pr{|\varphi_{t+1} - \varphi_{t}| \ge \eta \varphi_{t}}[X_t = x, Y_t = y] - \Pr{A}[X_t = x, Y_t = y]}{1-\Pr{A}[X_t = x, Y_t = y]} \\
		\label{eq:exp_potential:5}
		&\hspace{1cm}\ge \frac{\kappa - \Pr{A}[X_t = x, Y_t = y]}{1-\Pr{A}[X_t = x, Y_t = y]} .
	\end{align}
	
	By combining \cref{eq:exp_potential:4,eq:exp_potential:5} and observing that for $\varphi_{t} \ge 1$ it holds that $\varphi_{t} = \eulerE^{\psi_{t}}$, we get
	\begin{align}
		\notag
		&\E{(\psi_{t+1} - \psi_{t})^2}[X_t = x, Y_t = y, \overline{A}] \\
		\notag
		&\hspace{1cm}\ge \ln(1+\eta)^2 \Pr{|\psi_{t+1} - \psi_{t}| \ge \ln(1+\eta)}[X_t = x, Y_t = y, \overline{A}] \\
		\label{eq:exp_potential:6}
		&\hspace{1cm}\ge \ln(1+\eta)^2 \frac{\kappa - \Pr{A}[X_t = x, Y_t = y]}{1-\Pr{A}[X_t = x, Y_t = y]} .
	\end{align}
	
	Last, we use \cref{eq:exp_potential:1,eq:exp_potential:6} and that $\eta < 1$ implies $\ln(1 + \eta) \le \ln(2)$ to obtain
	\begin{align*}
		&\E{(\psi_{t+1} - \psi_{t})^2}[X_t = x, Y_t = y]  \\
		&\hspace{0.9cm}\ge 4 \ln(2)^2 \Pr{A}[X_t = x, Y_t = y] \\
		&\hspace{0.9cm}\quad+ (1-\Pr{A}[X_t = x, Y_t = y]) \ln(1+\eta)^2 \frac{\kappa - \Pr{A}[X_t = x, Y_t = y]}{1-\Pr{A}[X_t = x, Y_t = y]} \\
		&\hspace{0.9cm}= 4 \ln(2)^2 \Pr{A}[X_t = x, Y_t = y] + \ln(1+\eta)^2 \kappa - \ln(1+\eta)^2 \Pr{A}[X_t = x, Y_t = y] \\
		&\hspace{0.9cm}\ge 3 \ln(2)^2 \Pr{A}[X_t = x, Y_t = y] + \ln(1+\eta)^2 \kappa \\
		&\hspace{0.9cm}\ge \ln(1+\eta)^2 \kappa ,
	\end{align*}
	which shows \cref{eq:varianceBound}.
	
	By \Cref{lemma:greenberg_stopping} and $\psi_{0} \le \ln(D/d)$, we get for all $x, y \in \Omega$ that
	\begin{align*}
		\E{T}[X_0 = x, Y_0 = y] \leq \frac{\big(\ln(D/d) + 2 \ln(2)\big)^2}{\ln(1+\eta)^2 \kappa}.
	\end{align*}
	This results in the desired mixing time bound of
	\[
		\mix{\markov}{\err} 
		\leq \frac{\big(\ln(D/d) + 2 \ln(2)\big)^2}{\ln(1+\eta)^2 \kappa}  \ln \left (\frac{1}{\err} \right ) .
		\qedhere
	\]
\end{proof}

\section{\namedAppendix{\nameref{sec:polymerDynamics}}} \label{appendix:polymerDynamics}
We state all proofs that are omitted in \Cref{sec:polymerDynamics}.  

\paragraph*{\proofOf{\Cref{lemma:convergence}}}
\convergenceLemma*	
\begin{proof}
	First, note that~\markov[\polymerModel] is irreducible, as there is a positive probability to go from any polymer family $\polyfam \in \polyfams$ to the empty polymer family~$\emptyset$ in a finite number of steps by consecutively removing each polymer $\gamma \in \polyfam$.
	Similarly, there is a positive probability to go from~$\emptyset$ to any polymer family $\polyfam' \in \polyfams$ in a finite number of steps by consecutively adding all polymers $\gamma' \in \polyfam'$.
	
	We proceed by proving that ~\gibbsDistribution[\polymerModel], which we abbreviate as~\gibbsDistribution, is a stationary distribution of~\markov[\polymerModel].
	To this end, we show that~\markov[\polymerModel] satisfies the detailed-balance condition with respect to~\gibbsDistribution.
	That is, for all $\polyfam, \polyfam' \in \polyfams$, it holds that 
	\begin{equation}
	\label{eq:detailedBalance}
		\gibbsDistributionFunction{\polyfam} \cdot \transitionProbabilities{\polyfam}{\polyfam'} = \gibbsDistributionFunction{\polyfam'} \cdot \transitionProbabilities{\polyfam'}{\polyfam}.
	\end{equation}
	Note that it is sufficient to check \cref{eq:detailedBalance} for all pairs of states with a symmetric difference of exactly one polymer.
	
	Let $\polyfam, \polyfam' \in \polyfams$ and assume without loss of generality that $\polyfam' = \polyfam \cup \{\gamma\}$ for some polymer $\polymer \notin \polyfam$.
	Note that, by \cref{eq:transitions}, $\transitionProbabilities{\polyfam}{\polyfam'} = \weight[\polymer] \cdot \transitionProbabilities{\polyfam'}{\polyfam}$.
	Further, by definition of the Gibbs distribution, we have $\gibbsDistributionFunction{\polyfam'} = \weight[\polymer] \cdot \gibbsDistributionFunction{\polyfam}$.
	Thus, we get
	\[
		\gibbsDistributionFunction{\polyfam} \cdot \transitionProbabilities{\polyfam}{\polyfam'}
		= \gibbsDistributionFunction{\polyfam} \cdot \weight[\polymer] \cdot \transitionProbabilities{\polyfam'}{\polyfam}
		= \gibbsDistributionFunction{\polyfam'} \cdot \transitionProbabilities{\polyfam'}{\polyfam},
	\]
	which shows that $\gibbs$ is a stationary distribution of ~\markov[\polymerModel].
	
	Finally, we argue that ~\markov[\polymerModel] is ergodic.	
	Note that an irreducible Markov chain has a stationary distribution if and only if it is positive recurrent.
	In addition, every state of ~\markov[\polymerModel] has a positive self-loop probability, which implies that the chain is aperiodic.
	This shows that ~\markov[\polymerModel] is ergodic and concludes the proof. 
\end{proof}	

\paragraph*{\proofOf{\Cref{lemma:mixing}}}
\mixingLemma*	
\begin{proof}
	We aim to apply \Cref{lemma:exp_potential}, which requires us to define a potential~$\delta$.
	%
	We do so by utilizing the function $\delta'\colon \polys \to \R_{> 0}$ with $\gamma \mapsto f(\gamma)/\big(z_{\gamma} (1 + w_{\gamma})\big)$.
	Let $\oplus$ denote the symmetric set difference.
	For all $\polyfam, \polyfam' \in \polyfams$, we define
	\[
		\delta(\polyfam, \polyfam') = \sum_{\gamma \in \polyfam \oplus \polyfam'} \delta'(\gamma).
	\]
	Note that $\delta(\polyfam, \polyfam')$ only depends on the symmetric difference of~$\polyfam$ and~$\polyfam'$ and that $\delta(\polyfam, \polyfam') = 0$ if and only if $\polyfam \oplus \polyfam' = \emptyset$, which only is the case when $\polyfam = \polyfam'$.
	
	We continue by constructing a coupling between two copies of~\markov[\polymerModel], namely between $(X_t)_{t \in \N}$ and $(Y_t)_{t \in \N}$.
	We couple these chains such that, for each transition,
	\begin{itemize}
		\item both choose the same index $i \in [m]$ and
		\item both draw the same polymer family $\polyfam_{\Delta} \in \polyfams[][\polymerClique[i]]$ from $\gibbsDistribution[][\polymerClique[i]]$.
	\end{itemize}
	This constitutes a valid coupling, as each chain transitions according to its desired marginal transition probabilities.
	
	We now show for all $t \in \N$ and $\polyfam, \polyfam' \in \polyfams$ that
	\begin{align*}
		\E{\delta(X_{t+1}, Y_{t+1})}[X_t = \polyfam, Y_t = \polyfam'] \le \delta(\polyfam, \polyfam').
	\end{align*}
	Note that this trivially holds if $\polyfam = \polyfam'$, as the chains~$X$ and~$Y$ behave identically from then on. Thus, we are left with the case that $\polyfam \neq \polyfam'$, which implies that $|\polyfam \oplus \polyfam'| \geq 1$.
	
	We introduce the following notation.
	For all $\gamma \in \polys$, let $N(\gamma) = \{\gamma' \in \polys \mid \gamma' \ncomp \gamma, \gamma' \neq \gamma\}$ denote the neighborhood of~$\gamma$.
	We extend this definition to arbitrary subsets of polymers $\restrictedPolys \subseteq \polys$ by $N(\restrictedPolys) = \bigcup_{\gamma \in \restrictedPolys} N(\gamma)$.
	
	Let $\Delta = \polyfam \oplus \polyfam'$, and let $\gamma \in \Delta$.
	Assume without loss of generality that $\gamma \in \polyfam$.
	By \cref{eq:transitions}, with probability $z_{\gamma}/m$, the chain~$X$ removes~$\gamma$, causing $\delta(X_{t+1}, Y_{t+1})$ to decrease by at least $\delta'(\gamma)$.
	Further, if $\gamma \in \Delta \setminus N(\Delta)$, with probability $w_{\gamma} z_{\gamma}/m$, the chain~$Y$ adds~$\gamma$ and the chain~$X$ remains in its state.
	Again, $\delta(X_{t+1}, Y_{t+1})$ decreases by $\delta'(\gamma)$.
	Note that the case $\gamma \in \polyfam'$ is exactly symmetric with $X$ and $Y$ swapped.
	%
	
	Let $\delta^-(\polyfam, \polyfam')$ denote the expected (conditional) decrease of~$\delta$.
	By the observations above, we see that
	\[
		\delta^-(\polyfam, \polyfam')
		= \sum_{\gamma \in \Delta} \delta'(\gamma) \frac{z_{\gamma}}{m} + \sum_{\gamma \in \Delta \setminus N(\Delta)} \delta'(\gamma) w_{\gamma }\frac{z_{\gamma}}{m}	
		= \sum_{\gamma \in \Delta} \delta'(\gamma) \frac{z_{\gamma}}{m} (1 + w_{\gamma}) - \sum_{\gamma \in \Delta \cap N(\Delta)} \delta'(\gamma) w_{\gamma }\frac{z_{\gamma}}{m}.
	\]
	
	Moreover, $\delta$ increases whenever a polymer~$\gamma$ is added to only one of both chains.
	This only occurs if $\gamma \in N(\Delta) \setminus \Delta$ and has probability $w_{\gamma} z_{\gamma}/m$ for each such polymer.
	Similarly to the expected decrease, we denote the expected increase by $\delta^+(\polyfam, \polyfam')$. We bound
	\begin{align*}
		&\delta^+(\polyfam, \polyfam')		
		\le \sum_{\gamma \in N(\Delta) \setminus \Delta} \delta'(\gamma) w_{\gamma} \frac{z_{\gamma}}{m}	
		= \sum_{\gamma \in N(\Delta)} \delta'(\gamma) w_{\gamma} \frac{z_{\gamma}}{m} 
		- \sum_{\gamma \in \Delta \cap N(\Delta)} \delta'(\gamma) w_{\gamma} \frac{z_{\gamma}}{m} \\
		&\quad \le \sum_{\gamma \in \Delta} \sum_{\gamma' \in N(\gamma)} \delta'(\gamma') w_{\gamma'} \frac{z_{\gamma'}}{m}
		- \sum_{\gamma \in \Delta \cap N(\Delta)} \delta'(\gamma) w_{\gamma} \frac{z_{\gamma}}{m}.
	\end{align*}
	
	Together, we obtain
	\begin{align*}
		&\E{\delta(X_{t+1}, Y_{t+1})}[X_t = \polyfam, Y_t = \polyfam'] 
		= \delta(\polyfam, \polyfam') + \delta^+(\polyfam, \polyfam') - \delta^-(\polyfam, \polyfam') \\
		&\quad \le \delta(\polyfam, \polyfam') 
		+  \sum_{\gamma \in \Delta} \sum_{\gamma' \in N(\gamma)} \delta'(\gamma') w_{\gamma'}  \frac{z_{\gamma'}}{m} 
		- \sum_{\gamma \in \Delta} \delta'(\gamma) \frac{z_{\gamma}}{m} (1 + w_{\gamma}) \\
		&\quad = \delta(\polyfam, \polyfam') 
		+  \sum_{\gamma \in \Delta} 
		\left( \sum_{\gamma' \in N(\gamma)} \delta'(\gamma') w_{\gamma'} \frac{z_{\gamma'}}{m} 
		- \delta'(\gamma) \frac{z_{\gamma}}{m} (1 + w_{\gamma}) \right) .
	\end{align*}
	
	We proceed by showing that, for each $\gamma \in \Delta$, the respective summand in the sum above is at most zero.
	By the definition of~$\delta'$, we get
	\[
		\sum_{\gamma' \in N(\gamma)} \delta'(\gamma') w_{\gamma'} \frac{z_{\gamma'}}{m} 
		- \delta'(\gamma) \frac{z_{\gamma}}{m} (1 + w_{\gamma})
		= \frac{1}{m} \left( \sum_{\gamma' \in N(\gamma)} f(\gamma') \frac{w_{\gamma'}}{1 + w_{\gamma'}} 
		- f(\gamma) \right).
	\]
	By the definition of $N(\gamma)$ and since~\polymerModel satisfies the \pmc{}, we bound
	\[
		\frac{1}{m} \left( \sum_{\gamma' \in N(\gamma)} f(\gamma') \frac{w_{\gamma'}}{1 + w_{\gamma'}} 
		- f(\gamma) \right)
		= \frac{1}{m} \left( \sum_{\substack{\gamma' \in \polys\colon \gamma' \ncomp \gamma \\\gamma' \neq \gamma}} f(\gamma') \frac{w_{\gamma'}}{1 + w_{\gamma'}} 
		- f(\gamma) \right)
		\le 0.
	\]
	Consequently, we get that
	\[
		\E{\delta(X_{t+1}, Y_{t+1})}[X_t = \polyfam, Y_t = \polyfam'] \le \delta(\polyfam, \polyfam').
	\]
	
	We now show that there are values $\eta, \kappa \in (0, 1)$ such that, for all $t \in \N$ and all $\polyfam, \polyfam' \in \polyfams$ with $\polyfam \neq \polyfam'$, it holds that
	\begin{align}
	\label{eq:probabilityBound}
		\Pr{|\delta(X_{t+1}, Y_{t+1}) - \delta(\polyfam, \polyfam')| \ge \eta \delta(\polyfam, \polyfam')}[X_t = \polyfam, Y_t = \polyfam'][\big]
		\ge \kappa.
	\end{align}
	
	Note that every polymer family in~$\polyfams$ has at most~$m$ polymers because it can have at most one polymer from each polymer clique.
	Thus, for $\Delta = \polyfam \oplus \polyfam'$, we bound $|\Delta| \le 2m$.
	Consequently, there is at least one polymer $\gamma \in \Delta$ such that $\delta'(\gamma) \ge \delta(\polyfam, \polyfam')/(2m)$.
	Assume without loss of generality that $\gamma \in \polyfam$.
	With probability $z_{\gamma}/m$, chain~$X$ deletes~$\gamma$, resulting in $\delta$ decreasing by at least $|\delta(X_{t+1}, Y_{t+1}) - \delta(\polyfam, \polyfam')| \geq \delta'(\gamma) \ge \delta(\polyfam, \polyfam') / (2m)$.
	Thus, \cref{eq:probabilityBound} is true for $\eta = 1/(2m)$ and $\kappa = z_{\gamma}/m \ge \big(\min_{\gamma \in \polys} \{z_{\gamma}\}\big)/m$.
	
	It remains is to determine $d, D \in \R_{> 0}$ such that, for all $\polyfam, \polyfam' \in \polyfams$ with $\polyfam \neq \polyfam'$, it holds that $\delta(\polyfam, \polyfam') \in [d, D]$.
	Let $\Delta = \polyfam \oplus \polyfam'$, noting again that $|\Delta| \le 2m$.
	We choose
	\begin{align*}
		d & \ge \min_{\gamma \in \polys} \{\delta'(\gamma)\} 
		= \min_{\gamma \in \polys} \left\{\frac{f(\gamma)}{z_{\gamma} (1 + w_{\gamma})}\right\} \ \textrm{ and} \\
		D & \le 2 m \max_{\gamma \in \polys} \{\delta'(\gamma)\} 
		= 2m \max_{\gamma \in \polys} \left\{\frac{f(\gamma)}{z_{\gamma} (1 + w_{\gamma})}\right\}.
	\end{align*}
	
	Applying \Cref{lemma:exp_potential} and observing that $\ln \left (1 + \frac{1}{2m} \right )^2 \ge \frac{1}{4 m^2}$ concludes the proof.
\end{proof}

\paragraph*{\proofOf{\Cref{prop:comparison}}}
\comparisonProposition*
\begin{proof}
	Instead of directly working with the \pmc, we show that the \fpc implies that
	\[
		\sum_{\polymer' \in \polys\colon \polymer' \ncomp \polymer} f(\polymer') \weight[\polymer'] \le f(\polymer), 
	\]
	for all polymers $\polymer \in \polys$.
	By $\weight[\polymer] > \frac{\weight[\polymer]}{1 + \weight[\polymer]}$, we can immediately conclude that \pmc is satisfied as well.
	
	Note that $\emptyset \in \polyfams[\polymerModel][\neighbors{\polymer}]$ and, for all $\polymer' \in \polys$ with $\polymer' \ncomp \polymer$, it holds that $\{\polymer'\} \in \polyfams[\polymerModel][\neighbors{\polymer}]$.
	Thus,
	\[
		\sum_{\polymer' \in \polys\colon \polymer' \ncomp \polymer} f(\polymer') \weight[\polymer']
		< 1 + \sum_{\polymer' \in \polys\colon \polymer' \ncomp \polymer} f(\polymer') \weight[\polymer']
		\le \sum_{\polyfam \in \polyfams[\polymerModel][\neighbors{\polymer}]} \prod_{\polymer' \in \polyfam} f(\polymer') \weight[\polymer'] 
		\le f(\polymer) ,
	\]
	which proves the claim.
\end{proof}

\section{\namedAppendix{\nameref{sec:algo}}} \label{appendix:algo}
Here we present all proofs that are absent from \Cref{sec:algo}.

\paragraph*{\proofOf{\Cref{thm:sampling}}}
\samplingTheorem*
\begin{proof}
	In order to sample from~$\gibbs$, we utilize the polymer Markov chain \markov[\polymerModel] based on~\polymerClique.
	By \Cref{thm:markov_chain}, it holds that
	\[
		\mix{\markov[\polymerModel]}{\err} 
		\in \bigO{
			m^3 \prt_{\max}
			\ln \left( 
			m^2 \prt_{\max}^2
			\frac{f_{\max}}{f_{\min}}
			\right)^2
			\ln \left( \frac{1}{\err} \right)
		}.	
	\]
	Due to \cref{thm:sampling:prt,thm:sampling:pmc}, it holds that $\mix{\markov[\polymerModel]}{\err} \in \poly{\numberOfCliques/\err}$.
	It remains to show that each step of~\markov[\polymerModel], as laid out in \Cref{def:markov_chain}, can be computed in time $\poly{\numberOfCliques}$.
	To this end, let $X_t$ denote the current state of $\poly{\numberOfCliques}$.
	
	Because of \cref{assume:sampling:cover,thm:sampling:inner}, for all $i \in [\numberOfCliques]$, we can draw~$\clique_i$ uniformly at random and can sample $\polyfam \in \polyfams[][\polymerClique[i]]$ according to $\gibbsDistribution[][\polymerClique[i]]$ in time $\poly{\numberOfCliques}$. This covers \cref{line:choosePolymerClique,line:samplePolymer}.
	
	Regarding \Cref{line:removePolymer}, note that we can check whether $\polyfam = \emptyset$ in time $\poly{\numberOfCliques}$.
	Assume that $\polyfam = \emptyset$, and note that $|X_t| \le m \in \poly{m}$, as~$X_t$ contains at most one polymer per polymer clique.
	In order to compute $X_t \setminus \clique_i$, it suffices to iterate over every $\gamma \in \polyfam$ and check if $\gamma \in \clique_i$, which can be done in time $\poly{\numberOfCliques}$, by \cref{assume:sampling:clique}. Once we found $\gamma \in \clique_i$, we remove it in time $\poly{\numberOfCliques}$.
	
	Regarding \Cref{line:addPolymer}, assume now that $\polyfam = \{\gamma\}$ for some $\gamma \in \clique_i$.
	In order to decide if $X_t \cup \polyfam$ is a valid polymer family, it is sufficient to iterate over all $\gamma' \in X_t$ and check whether any of them is incompatible to $\gamma$.
	By \cref{assume:sampling:ncomp} this can be done in time $\poly{\numberOfCliques}$, which concludes the proof.
\end{proof}

\paragraph*{Proof of \Cref{lemma:appx_algo}}
\appxAlgoLemma*
\begin{proof}
		Let $i \in [\numberOfCliques]$.
	We start by bounding $\prtFrac[i]$.
	Note that $\partitionFunction[][\Uclique{i}] \ge \partitionFunction[][\Uclique{i-1}]$ and $\partitionFunction[][\Uclique{i}] \le \partitionFunction[][\Uclique{i-1}] \partitionFunction[][\polymerClique[i]]$.
	Thus, $1/\prt_{\max} \le \prtFrac[i] \le 1$.
	
	The remaining proof is split into two parts.
	First, we bound $\E{\appxPrtFrac}$ with respect to $1/\prt$.
	Second, we bound the absolute difference of $\appxPrtFrac$ and $\E{\appxPrtFrac}$.
	Combining both errors concludes the proof.
	
	\medskip
	\emph{Bounding $\E{\appxPrtFrac}$.}
	Note that, for all $i \in [\numberOfCliques]$, it holds that $\prtFrac[i] - \err_{\numSamples} \le \E{\appxPrtFrac[i]} \le \prtFrac[i] + \err_{\numSamples}$, since~\appxPrtFrac[i] is the mean of $\err_\numSamples$-approximate samples.
	By the bounds on $\prtFrac[i]$ and our choice of~$\err_{\numSamples}$, we get
	\[
		\left( 1 - \frac{\err}{5 \numberOfCliques} \right) \prtFrac[i] \le \E{\appxPrtFrac[i]} \le \left( 1 + \frac{\err}{5 \numberOfCliques} \right) \prtFrac[i] .
	\]
	
	Recall that $1/\prt = \prod_{i \in [\numberOfCliques]} \prtFrac[i]$. Further, since $\{\appxPrtFrac[i]\}_{i \in [\numberOfCliques]}$ are mutually independent, we have $\E{\appxPrtFrac} = \prod_{i \in [\numberOfCliques]} \E{\appxPrtFrac[i]}$.
	Consequently, since, for all $x \in [0, 1]$ and all $k \in \N_{> 0}$, it holds that $\eulerE^{-x/k} \leq 1 - x/(k + 1)$~\cite[Chapter $3$]{jerrum2003counting}, we obtain
	\begin{align} \label{lemma:appx_algo:eq1}
		\eulerE^{-\err/4} \frac{1}{\prt}
		\le \left( 1 - \frac{\err}{5 \numberOfCliques} \right)^{\numberOfCliques} \frac{1}{\prt}
		\le \E{\appxPrtFrac}
		\le \left( 1 + \frac{\err}{5 \numberOfCliques} \right)^{\numberOfCliques} \frac{1}{\prt}
		\le \eulerE^{\err/5} \frac{1}{\prt}
		\le \eulerE^{\err/4} \frac{1}{\prt} .
	\end{align}		
	
	\medskip
	\emph{Bounding the absolute difference of $\appxPrtFrac$ and $\E{\appxPrtFrac}$.} By Chebyshev's inequality, we get
	\begin{align*}
		\Pr{\absolute{\appxPrtFrac - \E{\appxPrtFrac}} \ge \frac{\err}{5} \E{\appxPrtFrac}}
		& \le \frac{25}{\err^2} \frac{\Var{\appxPrtFrac}}{\E{\appxPrtFrac}^2}
		= \frac{25}{\err^2} \left( \frac{\E{\appxPrtFrac^2}}{\E{\appxPrtFrac}^2} - 1 \right).
	\end{align*}
	Again, by the mutual independence of $\{\appxPrtFrac[i]\}_{i \in [\numberOfCliques]}$, we have $\E{\appxPrtFrac}^2 = \prod_{i \in [\numberOfCliques]} \E{\appxPrtFrac[i]}^2$ and $\E{\appxPrtFrac^2} = \prod_{i \in [\numberOfCliques]} \E{\appxPrtFrac[i]^2}$.
	Thus,
	\begin{align*}
		\Pr{\absolute{\appxPrtFrac - \E{\appxPrtFrac}} \ge \frac{\err}{5} \E{\appxPrtFrac}}
		\leq \frac{25}{\err^2} \left( \prod_{i \in [\numberOfCliques]} \frac{\E{\appxPrtFrac[i]^2}}{\E{\appxPrtFrac[i]}^2} - 1 \right)
		= \frac{25}{\err^2} \left( \prod_{i \in [\numberOfCliques]} \left( 1 + \frac{\Var{\appxPrtFrac[i]}}{\E{\appxPrtFrac[i]}^2} \right) - 1 \right) .
	\end{align*} 
	
	For bounding the variance of $\appxPrtFrac[i]$, recall that $\appxPrtFrac[i] = \frac{1}{\numSamples} \sum_{j \in [\numSamples]} \indicator{\polyfam^{(j)} \in \polyfams[][\Uclique{i-1}]}$, where $\{\polyfam^{(j)}\}_{j \in [\numSamples]}$ are independently drawn from an $\err_\numSamples$-approximation of $\gibbsDistribution[][\Uclique{i}]$.
	By \cref{eq:expectationOfIndicatorFunctionForApproximation}, we have
	\[
		\Var{\appxPrtFrac[i]} 
		= \frac{1}{\numSamples^2} \sum_{j \in [\numSamples]}  \Var{\indicator{\polyfam^{(j)} \in \polyfams[][\Uclique{i-1}]}} 
		= \frac{1}{\numSamples} \E{\appxPrtFrac[i]} (1-\E{\appxPrtFrac[i]}).
	\]
	Noting that $\E{\appxPrtFrac[i]} \ge \big( 1 - \err/(5 \numberOfCliques) \big) \prtFrac[i] \ge 4/(5 \prt_{\max})$, we bound
	\[
		\frac{\Var{\appxPrtFrac[i]}}{\E{\appxPrtFrac[i]}^2} 
		= \frac{1}{\numSamples \E{\appxPrtFrac[i]}} - \frac{1}{\numSamples} 
		\le \frac{5 \prt_{\max}}{4 \numSamples} .
	\]
	Hence, using that, for all $x \in [0, 1]$ and all $k \in \N_{> 0}$, it holds that $\eulerE^{x/(k+1)} \le 1 + x/k$, we obtain
	\begin{align*}
		&\Pr{|\appxPrtFrac - \E{\appxPrtFrac}| \ge \frac{\err}{5} \E{\appxPrtFrac}}
		\le \frac{25}{\err^2} \left( \left( 1 + \frac{5 \prt_{\max}}{4 \numSamples} \right)^{\numberOfCliques} - 1 \right)
		\le \frac{25}{\err^2} \left( \eulerE^{5 \prt_{\max} \numberOfCliques/(4 \numSamples)} - 1 \right) \\
		&\quad \le \frac{25}{\err^2} \frac{5 \prt_{\max} \numberOfCliques}{4 \numSamples - 1} .
	\end{align*}
	Due to our choice of~\numSamples, and using the same approach as in bounding \cref{lemma:appx_algo:eq1}, with probability at least $3/4$, it holds that
	\begin{align} \label{lemma:appx_algo:eq2}
		\eulerE^{-\err/4} \E{\appxPrtFrac} \leq \left( 1 - \frac{\err}{5} \right) \E{\appxPrtFrac} \le \appxPrtFrac \le \left( 1 + \frac{\err}{5} \right) \E{\appxPrtFrac} \leq \eulerE^{\err/4} \E{\appxPrtFrac}.
	\end{align}
	
	\medskip
	\emph{Combining the results.}
	Combining \cref{lemma:appx_algo:eq1,lemma:appx_algo:eq2} yields that
	\begin{align*}
		(1 - \err) {\prt} \le \eulerE^{-\err/2} {\prt} \le \frac{1}{\appxPrtFrac} \le  \eulerE^{\err/2} {\prt} \le  (1+\err) {\prt}
	\end{align*}
	with probability at least $3/4$, which concludes the proof.
\end{proof}

\paragraph*{\proofOf{\Cref{thm:appx_partition_function}}}
\approximationPartitionFunction*
\begin{proof}
	The statement follows from \Cref{lemma:appx_algo}, choosing the parameters of \Cref{algo:appx_prt} accordingly.
	Note that \Cref{thm:sampling} both assume that $\prt_{\max} \in \poly{\numberOfCliques}$.
	This implies that $\numSamples \in \poly{\numberOfCliques/\err}$ and that we can sample $\err_{\numSamples}$-approximately from $\gibbsDistribution$ in time $\poly{\numberOfCliques/\err}$. 
	Note that, for all $i \in [\numberOfCliques]$, the same holds for $\gibbsDistribution[][\Uclique{i}]$, as this only requires the Markov chain to ignore some of the polymer cliques in each step.
\end{proof}

\section{\namedAppendix{\nameref{sec:truncation}}} \label{appendix:truncation}
Here we give all proofs that are omitted in \Cref{sec:truncation}.

\paragraph*{\proofOf{\Cref{lemma:clique_truncation}}}
\cliqueTruncationLemma*
\begin{proof}
	We start by proving $\eulerE^{- \err} \le \prt_{\le k}/\prt \le 1$.
	Since $\prt_{\le k} \le \prt$, as removing polymers does not increase the partition function, it remains to show that $\prt \le \eulerE^{\err} \prt_{\le k}$.
	
	We observe that $\prt \le \prt_{\le k} \prt_{> k}$ with equality if and only if, for all $\polyfam \in \polyfams_{\le k}$ and all $\polyfam' \in \polyfams_{>k}$, it holds that $\polyfam \cup \polyfam' \in \polyfams$.
	We proceed by showing that $\prt_{> k} \le \eulerE^{\err}$.
	
	Note that $\polys^{>k} = \bigcup_{i \in [\numberOfCliques]} \clique_i^{> k}$ and that each polymer family in $\polyfams_{>k}$ contains at most one polymer from each $\clique_i^{> k}$.
	Thus, we obtain
	\[
		\prt_{> k} 
		\le \prod_{i \in [\numberOfCliques]} \partitionFunction[][\clique_i^{> k}]
		= \prod_{i \in [\numberOfCliques]} \left( 1 + \sum_{\gamma \in \clique_i^{> k}} w_{\gamma} \right) .
	\]
	
	Due to \cref{eq:truncationLemmaAssumption}, we get $\prt_{> k} \le ( 1 + \err/m )^m \le \eulerE^{\err}$,	which proves the first claim.
	
	Observing that
	\[
		\dtv{\gibbs}{\gibbs_{\le k}} = \frac{\prt - \prt_{\le k}}{\prt} \le 1 - \eulerE^{- \err} \le \err 
	\]
	proves the second claim.
\end{proof}

\paragraph*{\proofOf{\Cref{lemma:ctc}}}
\ctcLemma*
\begin{proof}
    Let $\err' \in (0, 1)$ and $k \ge g^{-1} \left( B/\err' \right)$.
	Due the \ctc{} and the monotonicity of~$g$, we observe that
	\begin{align*}
		g(k) \sum_{\gamma \in \clique_i^{>k}} w_{\gamma} \le \sum_{\gamma \in \clique_i^{>k}} g(\size{\gamma}) w_{\gamma} \le \sum_{\gamma \in \clique_i} g(\size{\gamma}) w_{\gamma} \le B.
	\end{align*}
	As $g$ is positive, dividing by $g(k)$ yields $\sum_{\gamma \in \clique_i^{>k}} w_{\gamma} \le B/g(k)$.
	Substituting our bound for~$k$ and noting that~$g$ is invertible, we conclude that
	\begin{equation*}
		\sum_{\gamma \in \clique_i^{>k}} w_{\gamma} 
		\le \frac{B}{g \left( g^{-1} \left( \frac{B}{\err'} \right) \right)}
		= \err',
	\end{equation*}
	which proves the claim.
\end{proof}

\paragraph*{\proofOf{\Cref{thm:sampling_trunc}}}
\samplingTruncTheorem*
\begin{proof}
	As in the proof of \Cref{thm:sampling}, we consider the polymer Markov chain~\markov[\polymerModel].
	Further, let $k = g^{-1}(2Bm/\err)$, let $\markov_{k}$ denote the polymer Markov chain on $(\polys_{\le k}, w, \ncomp)$, and let~$P_k$ denote its transitions.
	We aim to run $\markov_{k}$ for at least $t^* = \mix{\markov}{\err/2}$ iterations, starting from $\emptyset \in \polyfams_{\le k}$.
	
	We prove that $\dtv{\gibbs}{P_k^{t^*} (\emptyset, \cdot)} \le \err$.
	By the triangle inequality, we obtain 
	\[
		\dtv{\gibbs}{P_k^{t^*} (\emptyset, \cdot)} \le \dtv{\gibbs}{\gibbs_{\le k}} + \dtv{\gibbs_{\le k}}{P_k^{t^*} (\emptyset, \cdot)}.	
	\]
	By our choice of $k$ and by \Cref{lemma:model_truncation} together with \cref{thm:samplingTrunc:ctc}, we get that $\dtv{\gibbs}{\gibbs_{\le k}} \le \err/2$.
	Further, note that truncation preserves the \pmc{} for the same function $f$ and does not increase any quantity that is used for bounding the mixing time.
	Thus, $\mix{\markov_k}{\err/2} \le \mix{\markov[\polymerModel]}{\err/2}$, and we obtain $\dtv{\gibbs_{\le k}}{P_k^{t^*} (\emptyset, \cdot)} \le \err/2$ for our choice of $t^*$.
	
	It remains to show that the runtime is bounded by $\poly{\numberOfCliques/\err}$.
	Analogously to the proof of \Cref{thm:sampling}, due to \cref{thm:samplingTrunc:prt,thm:samplingTrunc:pmc}, we know that $\mix{\markov[\polymerModel]}{\err/2} \in \poly{\numberOfCliques/\err}$, which implies $\mix{\markov_k}{\err/2} \in \poly{\numberOfCliques/\err}$.
	Also analogously, it holds that each step can be done in $\poly{\numberOfCliques}$, except for sampling, for all $i \in [\numberOfCliques]$, from~$\gibbsDistribution[][\polymerClique[i]]$.
	However, note that, for all $i \in [\numberOfCliques]$, we only need to sample from $\gibbsDistribution[][\clique_i^{\le k}]$.
	We do so by enumerating $\clique_i^{\le k}$ in time $t(k)$.
	By our choice of $k$ and by \cref{thm:samplingTrunc:ctc}, this takes time at most
	\[
		t(k) = t \left( g^{-1} \left( \frac{2 B m}{\err} \right) \right) \in \poly{\frac{\numberOfCliques}{\err}},	
	\] 
	which proves that we can $\err$-approximately sample from $\gibbs$ in the desired runtime.
	
	Showing that we can $\err$-approximate $\prt$ in time $\poly{\numberOfCliques/\err}$ is done analogously.
	By \Cref{lemma:model_truncation} and \cref{thm:samplingTrunc:ctc}, we know that for $k = g^{-1}(2Bm/\err)$ it holds that $\eulerE^{- \err / 2} \le \prt_{\le k}/\prt \le 1$, which implies $\eulerE^{- \err / 2} \prt \le \prt_{\le k} \le \prt$.
	As argued above, the truncation of the polymer model $\polymerModel$ to this size~$k$ satisfies the conditions of \Cref{thm:sampling}, where the sampling from each clique is done by ignoring polymers larger than $k$.
	Thus, by \Cref{thm:appx_partition_function}, we obtain an $\err/2$-approximation for $\prt_{\le k}$ in time $\poly{2 \numberOfCliques/\err} = \poly{\numberOfCliques/\err}$.
	Noting that, for $\err \le 1$, it holds that
	\[
		1 - \err \le \left( 1 - \frac{\err}{2} \right) \eulerE^{- \err / 2}
		\text{ and }
		\left( 1 + \frac{\err}{2} \right) \le 1 + \err,
	\]
	which concludes the proof.
\end{proof}

\section{\namedAppendix{Hard-Core Model on Bipartite Expanders}}
\label{appendix:expanders}

In order to demonstrate how \Cref{thm:sampling_trunc} improves known bounds for the algorithmic use of polymer models, we investigate the hard-core model for high fugacity $\fugacity \in \R_{>0}$ on bipartite $\alpha$-expanders with bounded maximum degree $\Degree$.
For a graph $(V, E)$ and an $S \subseteq V$, let $\neighbors{S}[G]$ denote the set of all vertices that are adjacent to a vertex in~$S$.
\begin{definition}[bipartite $\alpha$-expander]
	Let $G=(V, E)$ be a bipartite graph with partition $V=\partition{\leftPart} \cup \partition{\rightPart}$.
	For all $i \in \{\leftPart, \rightPart\}$, we call $S \subseteq \partition{i}$ \emph{small} if and only if $\absolute{S} \le \absolute{\partition{i}}/2$.
	For all $\alpha \in (0, 1)$, graph~$G$ is a \emph{bipartite $\alpha$-expander} if and only if, for all small sets of vertices~$S$, it holds that $\absolute{\neighbors{S}[G]} \ge (1 + \alpha) \absolute{S}$. 
\end{definition}  

For any graph $G$, the hard-core partition function is a graph polynomial of some parameter $\fugacity \in \R_{>0}$, called fugacity.
Let $\mathcal{I}_{G}$ be the set of all independent sets in $G$.
The hard-core partition function for fugacity $\fugacity$ is now formally defined as
\[
\hcPartitionFunction{\fugacity}[G] = \sum_{I \in \mathcal{I}_{G}} \fugacity^{\absolute{I}} .
\] 

We approximate $\hcPartitionFunction{\fugacity}[G]$ in terms of the partition function of two polymer models, constructed as proposed by Jenssen et~al.~\cite{JKP19}.
For a bipartite $\alpha$-expander $G$ with bounded degree $\Degree$, we consider the graph $G^2$, which is the graph with vertices $V$ and an edge between $v, u \in V$ if $v, u$ have at most distance $2$ in $G$.
For all $i \in \{\leftPart, \rightPart\}$, we define a polymer model $\polymerModelPart{i} = (\polysPart{i}, \weightPart{i}, \ncomp)$ as follows:
\begin{itemize}
	\item each polymer $\polymer \in \polysPart{i}$ is defined by a non-empty set of vertices $\polyVertex{\polymer} \subseteq \partition{i}$ such that $\polyVertex{\polymer}$ is small and induces a connected subgraph in $G^2$,
	\item for $\polymer \in \polysPart{i}$, let $\weightPart{i}[\polymer] = \fugacity^{\absolute{\polyVertex{\polymer}}} / \big( (1+\fugacity)^{\absolute{ \neighbors{\polyVertex{\polymer}}[G] }}\big)$, and
	\item two polymers $\polymer, \polymer'  \in \polysPart{i}$ are incompatible if and only if there are vertices $v \in \polyVertex{\polymer}, w \in \polyVertex{\polymer'}$ with graph distance at most~$1$ in~$G^2$.
\end{itemize}
To ease notation, for all $i \in \{\leftPart, \rightPart\}$, we write~$\partGibbs{i}$ and $\partPrt{i}$ instead of~$\gibbsDistribution[\polymerModelPart{i}]$ and~$\partitionFunction[\polymerModelPart{i}]$, respectively.

We use~$\polymerModelPart{\leftPart}$ and~$\polymerModelPart{\rightPart}$ for approximating the hard-core partition function of bipartite $\alpha$-expanders in the following sense.
\begin{lemma}[{\cite[Lemma~$19$]{JKP19}}]
	\label{lemma:hc_alpha}
	Given a bipartite $\alpha$-expander $G=(\partition{\leftPart} \cup \partition{\rightPart}, E)$ with $\absolute{\partition{\leftPart} \cup \partition{\rightPart}} = n$, let $\hcPartitionFunction{\fugacity}[G]$ denote its hard-core partition function with fugacity $\fugacity \in \R_{>0}$, and let the polymer models $\polymerModelPart{\leftPart}, \polymerModelPart{\rightPart}$ be defined as above.
	For all $\fugacity \ge \eulerE^{11/\alpha}$, it holds that
	\[
	(1 - \eulerE^{-n}) \hcPartitionFunction{\fugacity}[G] 
	\le \left( 1+\fugacity \right)^{\absolute{\partition{\rightPart}}} \partPrt{\leftPart} 
	+ \left( 1+\fugacity \right)^{\absolute{\partition{\leftPart}}} \partPrt{\rightPart}
	\le (1 + \eulerE^{-n}) \hcPartitionFunction{\fugacity}[G] .
	\]
\end{lemma}

To apply \Cref{thm:sampling_trunc}, we have to fix a polymer clique cover $\clique$ for each polymer model $\polymerModelPart{i}$ with $i \in \{\leftPart, \rightPart\}$.
Based on the incompatibility relation, a natural choice is to define, for each $v \in \partition{i}$, a clique~$\polymerClique[v]$ such that $\polymer \in \polymerClique[v]$ if and only if $v \in \polyVertex{\polymer}$.
As we need to verify the \pmc, it is useful to have a bound on the number of incompatible polymers, which the following lemma provides.
\begin{lemma}[{\cite[Lemma~$2.1$]{borgs2013left}}]
	\label{lemma:count_vertex}
	For an undirected graph $G=(V, E)$ with maximum degree~$\Degree$ and for all $v \in V$, the number of vertex-induced connected subgraphs that contain~$v$ and have at most $k \in \N_{>0}$ vertices is bounded from above by $\eulerE^k \Degree^{k-1}/\big(k^{3/2} \sqrt{2 \circlePi}\big)$.
\end{lemma}
Commonly, the bound $(\eulerE \Degree)^{k-1}/2$ is applied, as it is more convenient to work with.
However, this bound actually only holds for $k \ge 2$. 
Further, note that the original paper used a weaker bound, namely $(\eulerE \Degree)^{k}$.
Although this bound holds for all $k \in \N_{>0}$, it yields a much worse dependency on $\Degree$.
For a fair comparison, we added the result of refined calculations for the approach by Jenssen et~al.~\cite{JKP19} to \Cref{table:bounds}. 

In order to apply truncation, we further need a notion of size for polymers.
An obvious choice is to set $\size{\polymer} = \absolute{\polyVertex{\polymer}}$.
The following lemma then bounds the time for enumerating polymers in a clique up to some size $k \in \N_{>0}$.
\begin{lemma}[{\cite[Lemma~$3.7$]{PR17}}]
	\label{lemma:enum_vertex}
	Let $G=(V, E)$ be an undirected graph with maximum degree~$\Degree$, and let $v \in V$.
	There is an algorithm that enumerates all connected, vertex-induced sugraphs of $G$ that contain $v$ and have at most $k \in \N_{>0}$ vertices in time $\eulerE^{\bigO{ k \log(\Degree)}}$.
\end{lemma}

We now prove our bound on $\fugacity$ for an efficient approximation of the hard-core partition function on bipartite $\alpha$-expanders.
Most of the calculations are similar to those of Jenssen et~al.~\cite{JKP19}, except that we use our newly obtained conditions.
\begin{proposition}
	\label{thm:hc_appx}
	Let $G(\partition{\leftPart} \cup \partition{\rightPart}, E)$ be a bipartite $\alpha$-expander with $\absolute{\partition{\leftPart} \cup \partition{\rightPart}} = n$ and with maximum degree $\Degree \in \N_{>0}$.
	For $\fugacity \ge \max\{( \eulerE  \Degree^2/0.8 )^{1/\alpha}, \eulerE^{11/\alpha}\}$ and for all $\err \in (0, 1]$, there is an FPRAS for $\hcPartitionFunction{\fugacity}[G]$ with runtime $\left( n/\err \right)^{\bigO{\ln(\Degree)}}$.
\end{proposition}

\begin{proof}
	If $\err \in \bigO{\eulerE^{-n}}$, we compute $\hcPartitionFunction{\fugacity}[G]$ by enumerating all independent sets.
	Since there are at most $2^n$ independent sets, which is polynomial in $1/\eulerE^{-n}$, the statement then follows.
	It remains to analyze the case $\err \in \bigOmega{\eulerE^{-n}}$.
	To this end, assume that $\err \ge 4 \eulerE^{-n}$.
	
	By \Cref{lemma:hc_alpha}, $\hcPartitionFunction{\fugacity}[G]$ can be $\eulerE^{-n}$-approximated using $\partPrt{\leftPart}$ and $\partPrt{\rightPart}$.
	We aim for an $\err/4$-approximation of $\partPrt{\leftPart}$ and $\partPrt{\rightPart}$, each with failure probability at most $1 - \sqrt{3}/2$.
	Note that
	\[
	1 - \err \le \left( 1 - \eulerE^{-n}\right) \left( 1 - \frac{\err}{4} \right)
	\text{~~and~~}
	\left( 1 + \eulerE^{-n}\right) \left( 1 + \frac{\err}{4} \right) \le 1 + \err.
	\]
	Thus, with probability at least $(\sqrt{3}/2)^2 = 3/4$ the result is an $\err$-approximation of $\hcPartitionFunction{\fugacity}[G]$.
	We can obtain the desired error probability of at most $1 - \sqrt{3}/2$ for the approximations of $\partPrt{\leftPart}$ and $\partPrt{\rightPart}$ by taking the median of $\bigO{\ln ( 2/(2 - \sqrt{3}) )}[\big] = \bigO{1}$ independent approximations with failure probability at most $1/4$.
	
	Let $i \in \{\leftPart, \rightPart\}$.
	In order to approximate $\partPrt{i}$, we aim to apply \Cref{thm:sampling_trunc}.
	To this end, for all $v \in \partition{i}$, we define a polymer clique $\polymerClique[v]$ containing all polymers $\polymer \in \polysPart{i}$ with $v \in \polyVertex{\polymer}$.
	This results in a polymer clique cover of size $n$.
	
	We proceed by proving that the polymer model satisfies the \pmc for $f(\gamma) = \absolute{\polyVertex{\gamma}}$.
	We use \Cref{remark:spmc} to simplify this step. 
	This also implies that \cref{thm:samplingTrunc:prt} of \Cref{thm:sampling_trunc} is satisfied.
	For any $\polymer \in \polysPart{i}$ we start by bounding the set of polymers $\polymer' \ncomp \polymer$ by 
	\[
	\sum_{\polymer' \in \polysPart{i}\!\colon \polymer' \ncomp \polymer} f(\polymer') \weightPart{i}[\polymer']
	\le \sum_{v \in \neighbors{\polyVertex{\polymer}}[G^2]} \sum_{\polymer' \in \polymerClique[v]} f(\polymer') \weightPart{i}[\polymer']
	= \sum_{v \in \neighbors{\polyVertex{\polymer}}[G^2]} \sum_{k \in \N_{> 0}} \sum_{\substack{\polymer' \in \polymerClique[v] \\\absolute{\polyVertex{\polymer'}} = k}} f(\polymer') \weightPart{i}[\polymer'] .
	\]
	
	Because $G$ is a bipartite $\alpha$-expander, for all $\polymer \in \polysPart{i}$, we have $\weightPart{i}[\polymer] \le 1/\fugacity^{\alpha \absolute{\polyVertex{\gamma}}}$. 
	Further, note that the degree of $G^2$ is bounded by $\Degree^2$.
	By \Cref{lemma:count_vertex} and our definition of $f$, we obtain
	\[
	\sum_{v \in \neighbors{\polyVertex{\polymer}}[G^2]} \sum_{k \in \N_{> 0}} \sum_{\substack{\polymer' \in \polymerClique[v] \\\absolute{\polyVertex{\polymer'}} = k}} f(\polymer') \weightPart{i}[\polymer']	
	\le \Degree^2 \absolute{\polyVertex{\polymer}} \sum_{k \in \N_{> 0}} \frac{\eulerE^k \left( \Degree^2 \right)^{k-1}}{k^{3/2} \sqrt{2 \pi} } \cdot k \cdot \frac{1}{\fugacity^{\alpha k}}
	= \frac{\absolute{\polyVertex{\polymer}}}{\sqrt{2 \pi}} \sum_{k \in \N_{> 0}} \left( \frac{\eulerE \Degree^2}{\fugacity^{\alpha}} \right)^k \frac{1}{\sqrt{k}} .
	\]
	
	For $\fugacity \ge \left( \eulerE \Degree^2/0.8  \right)^{1/\alpha}$, we get
	\[
	\frac{\absolute{\polyVertex{\polymer}}}{\sqrt{2 \pi}} \sum_{k \in \N_{> 0}} \left( \frac{\eulerE \Degree^2}{\fugacity^{\alpha}} \right)^k \frac{1}{\sqrt{k}}
	\le \frac{\absolute{\polyVertex{\polymer}}}{\sqrt{2 \pi}} \sum_{k \in \N_{> 0}} \left( 0.8 \right)^k \frac{1}{\sqrt{k}}
	\le \frac{\absolute{\polyVertex{\polymer}}}{\sqrt{2 \pi}} \sqrt{2 \pi}
	= f(\polymer) .
	\]
	
	It remains to show, for all $v \in \partition{i}$, that $\polymerClique[v]$ satisfies the clique truncation condition for a $g\colon \R \to \R_{> 0}$ and a $B \in \R_{>0}$.
	To this end, for all $\polymer \in \polysPart{i}$, let $\size{\polymer} = \absolute{\polyVertex{\polymer}}$,, let $g(\size{\polymer}) = \eulerE^{0.2 \size{\polymer}}$, and let $B = 1$.
	Analogously to our verification of the \pmc, we see, for all $v \in \partition{i}$, that
	\[
	\sum_{\polymer \in \polymerClique[v]} g(\size{\polymer}) \weightPart{i}[\polymer] 
	\le \frac{1}{\Degree^2 \sqrt{2 \pi}} \sum_{k \in \N_{> 0}} \left( \frac{\eulerE \Degree^2}{\fugacity^\alpha} \right)^k \frac{1}{k^{3/2}} \eulerE^{0.2 k}
	= \frac{1}{\Degree^2 \sqrt{2 \pi}} \sum_{k \in \N_{> 0}} \left( \frac{\eulerE^{1.2} \Degree^2}{\fugacity^\alpha} \right)^k \frac{1}{k^{3/2}} .		
	\]
	For $\fugacity \ge \left( \eulerE \Degree^2/0.8 \right)^{1/\alpha}$, we get
	\[
	\frac{1}{\Degree^2 \sqrt{2 \pi}} \sum_{k \in \N_{> 0}} \left( \frac{\eulerE^{1.2} \Degree^2}{\fugacity^\alpha} \right)^k \frac{1}{k^{3/2}}
	\le \frac{1}{\Degree^2 \sqrt{2 \pi}} \sum_{k \in \N_{> 0}} \left( 0.8 \eulerE^{0.2} \right)^k \frac{1}{k^{3/2}} 
	< \frac{1}{\Degree^2 \sqrt{2 \pi}} 2.2 
	\le B .
	\]
	
	Last, we bound the runtime of the FPRAS.
	By \Cref{lemma:enum_vertex}, we can enumerate each polymer clique up to size $k$ in time $t(k) \in \eulerE^{\bigO{ k \log(\Degree)}}$.
	As $g^{-1}\colon x \mapsto 5 \ln(x)$, we have $t \circ g^{-1}\colon x \mapsto x^{\bigO{\ln (\Degree)}}$, which is polynomial for $\Degree \in \bigTheta{1}$.
	For the runtime bound, note that we truncate to size $k = g^{-1}(n/\err)$.
	Thus, the time for computing each step of the polymer Markov chain is bounded by $t(k) = \left( n/\err \right)^{\bigO{\ln(\Degree)}}$, which dominates the runtime.
\end{proof}

\paragraph*{Choice of $f$ and Calculation of \Cref{table:bounds}}
Note that the choice of the function $f$ used in the \pmc is very sensitive to the bound on the number of subgraphs.
For the bound stated in \Cref{lemma:count_vertex}, it turns out that using $f(\polymer) = \absolute{\polyVertex{\polymer}}$ yields the best bounds on $\fugacity$ (see the proof of \Cref{thm:hc_appx} for details).
With this choice of $f$, the condition that we identified in \Cref{remark:spmc} is similar to the mixing condition of \cite[Definition~1]{CGGPSV19}, except that we do not require a strict inequality.
Further, note that such a choice of $f$ is not possible for the Koteck{\'y}--Preiss condition~\cite{1986:Kotecky:cluster_expansion_polymer_models}. 
If purely exponential bounds on the number of subgraphs are used, the best results are usually obtained by setting $f$ to take an exponential form.
A detailed understanding of how to choose $f$ might be of interest for applications to specific graph classes and other combinatorial structures.

The results for the remaining applications in \Cref{table:bounds} are derived via similar calculations.
For the Potts model on expander graphs and the hard-core model on unbalanced bipartite graphs, we use \Cref{lemma:count_vertex,lemma:enum_vertex} together with the same function $f$ for the \pmc as in the proof of \Cref{thm:hc_appx}.
For the perfect matching polynomial, we use the bounds for the number of polymers and for polymer enumeration that are stated by Casel et~al.~\cite{CFFGL19}, and we choose $f(\polymer) = \eulerE^{a \size{\polymer}}$ for $a \approx 0.2$.

\end{document}